\documentclass[10pt,a4paper]{article}

\usepackage{xr-hyper}
\usepackage{hyperref}

\usepackage{fullpage}
\usepackage{chngcntr} 

\usepackage{MnSymbol} 

\usepackage{enumerate}
\usepackage[usenames, dvipsnames]{xcolor} 
\usepackage{amsfonts} 
\usepackage{dsfont} 
\usepackage{amsmath} 
\usepackage{amsthm} 
\usepackage{bm} 
\usepackage{bbm} 
\usepackage{ifthen} 
\usepackage{mathtools} 
\usepackage{multirow}
\usepackage[section]{placeins} 
\usepackage{float} 
\usepackage{url}
\RequirePackage{ifdraft}

\usepackage{fancyhdr}
 
\makeatletter
\def\parsetime#1#2#3#4#5\empty{#1#2:#3#4}
\def\parsedate#1:20#2#3#4#5#6#7#8\empty{20#2#3/#4#5/#6#7-\parsetime#8\empty}
\def\moddate#1{\expandafter\parsedate\pdffilemoddate{#1}\empty}
\newcommand{\timestamp}{%
\ifdraft{\textcolor{red}{\moddate{\jobname.tex}}}{\today}}
\makeatother

\makeatletter
\def\ifdraft{\ifdim\overfullrule>\z@
  \expandafter\@firstoftwo\else\expandafter\@secondoftwo\fi}
\makeatother

\newtheorem{theorem}{Theorem}[section]
\newtheorem*{theorem*}{Theorem}

\newtheorem{remark}[theorem]{Remark}

\newtheorem{corollary}[theorem]{Corollary}

\newtheorem*{test*}{Test}
\numberwithin{theorem}{section}
\numberwithin{equation}{section}

\newcommand{\E}{\mathbb{E}}

\newcommand{\N}{\mathbb{N}}
\newcommand{\R}{\mathbb{R}}
\newcommand{\V}{\mathbb{V}}

\newcommand{\X}{\bm{X}}
\newcommand{\U}{\bm{U}}
\newcommand{\Prob}{\mathbb{P}}

\newcommand{\vx}{\bm{x}}
\newcommand{\vy}{\bm{y}}
\newcommand{\vu}{\bm{u}}

\newcommand{\vsx}{\bm{x}^{[N]}}
\newcommand{\vsu}{\bm{u}^{[N]}}

\newcommand{\vsX}{\bm{X}^{[N]}}
\newcommand{\vsU}{\bm{U}^{[N]}}

\newcommand{\Mskript}{\mathcal{M}}

\newcommand{\One}{\mathds 1}

\newcommand{\Cov}{\operatorname{Cov}}

\newcommand{\ii}{\mathrm{i}}
\newcommand{\ee}{\mathrm{e}}

\newcommand{\hN}{\mbox{}^{\scriptscriptstyle N}\kern-1.5pt}

\newcommand{\pairslr}[3]{\left#1 #2 \right#3}
\newcommand{\pairscs}[4]{{\csname#1l\endcsname #2} #3 {\csname#1r\endcsname #4}}

\newcommand{\pairs}[4][lr]{%
        \ifthenelse{\equal{#1}{}}{%
                #2 #3 #4}{%
                \ifthenelse{\equal{#1}{lr}}{\pairslr{#2}{#3}{#4}}{%
                        \pairscs{#1}{#2}{#3}{#4}}}}

\newcommand{\Expect}[2][lr]{\E\ifthenelse{\equal{#1}{lr}}{\!}{}\pairs[#1]{(}{#2}{)}}
\newcommand{\Var}[2][lr]{\V\ifthenelse{\equal{#1}{lr}}{\!}{}\pairs[#1]{(}{#2}{)}}
\newcommand{\cov}[3][lr]{\Cov\ifthenelse{\equal{#1}{lr}}{\!}{}\pairs[#1]{(}{#2,#3}{)}}

\newcommand{\normal}{\color{black}}

\newenvironment{keywords}{\textit{Keywords:}}{}
\newenvironment{amscode}{\textit{2010 MSC:}}{}

\begin{document}
\title{Copula versions of distance multivariance and dHSIC via the distributional transform -- a general approach to construct invariant dependence measures}
\author{Bj\"{o}rn B\"{o}ttcher\footnote{TU Dresden, Fakult\"at Mathematik, Institut f\"{u}r Mathematische Stochastik, 01062 Dresden, Germany, email: \href{mailto:bjoern.boettcher@tu-dresden.de}{bjoern.boettcher@tu-dresden.de}}}

\date{}
\maketitle

\begin{abstract}
The multivariate Hilbert-Schmidt-Independence-Criterion (dHSIC) and distance multivariance allow to measure and test independence of an arbitrary number of random vectors with arbitrary dimensions. Here we define versions which only depend on an underlying copula. The approach is based on the distributional transform, yielding dependence measures which always feature a natural invariance with respect to scalings and translations. Moreover, it requires no distributional assumptions, i.e., the distributions can be of pure type or any mixture of discrete and continuous distributions and (in our setting) no existence of moments is required.

Empirical estimators and tests, which are consistent against all alternatives, are provided based on a Monte Carlo distributional transform. In particular, it is shown that the new estimators inherit the exact limiting distributional properties of the original estimators. Examples illustrate that tests based on the new measures can be more powerful than tests based on other copula dependence measures.
\end{abstract}


\begin{amscode}
 Primary 62H20,   	
Secondary 	62H15  	
\end{amscode}

\section{Introduction}

The detection and quantification of dependencies is essential for almost every statistical analysis. Although this is an old topic there have been recently several new contributions for the case of testing independence of multiple variables. We will focus here specifically on two: Distance multivariance \cite{BoetKellSchi2019,Boet2020} and the multivariate Hilbert-Schmidt-Independence-Criterion (dHSIC) \cite{PfisBuehSchoPete2017}. Multivariance includes Pearson's correlation and the RV-coefficient \cite{RobeEsco1976} as limiting cases and unifies the bivariate distance covariance \cite{SzekRizzBaki2007} and the  Hilbert-Schmidt-Independence-Criterion \cite{GretBousSmolScho2005}. For the latter dHSIC provides an alternative multivariate extension. As in the special case of Pearson's correlation the values of these measures are influenced by the actual dependence structure and the marginal distributions. 

For standardized comparisons a removal of the influence of the marginal distributions is of interest, it provides the foundation for a direct comparison based on the values of a measure, rather than a comparison of the corresponding p-values. Attempts to define dependence measures based on copulas go back at least to \cite{Wolf1977}. A recent approach of \cite{GeneNesRemiMurp2019} uses a Cram\'er-von Mises type test statistic, i.e.\ it
is based on an $L^2$-distance of distribution functions.
This structurally corresponds to approaches which are based on $L^2$-distances of characteristic functions. In fact the estimator $T_n$ of \cite{GeneNesRemiMurp2019} corresponds formally to the approach of multivariance, and the estimator $S_n$ of \cite{GeneNesRemiMurp2019} corresponds formally to the approach of dHSIC \cite{PfisBuehSchoPete2017} (see also Section 3.3 of \cite{Boet2020}). Therefore we will consider in this paper both: dHSIC and distance multivariance.

Our approach is related to \cite{GeneNesRemi2013}, where a multivariate extension of Spearman's rho was considered and the so called distributional transform (see Section \ref{sec:disttrans} below for details) was a key tool, see also \cite{Nes2007} for the foundations.

A theoretic (population) approach to transform any dependence measure into a measure which only depends on the copula consists of two steps:
1. Transform the random variables such that their distribution function is nothing but the copula.
2. Use the new random variables in the measure.

A slight difficulty with this approach is the fact that for non-continuous random variables the corresponding copula is in general not unique. But note, by an appropriate choice of the transformation procedure the resulting copula becomes unique.

Thus a natural approach to the corresponding sample version of the measure is:
1. Transform the samples to samples of the copula.
2. Use the new samples in the measure. 
For practical applications this shows that an explicit estimation of the copula is superficial, one only requires a method to transform samples into samples of the copula.

A closely related approach transforms each margin by its distribution function. This is well established standard, see e.g.\ \cite[Section 1]{Remi2009} for the setting of distance covariance, \cite[Section 2.4]{ChakZhan2019} for joint distance covariance (which is closely related to total distance multivariance) and \cite{PoczGhahSchn2012} for the setting of HSIC (and e.g.\ \cite{RoyGoswMurt2017,RoyGhosGosw2019} for more recent variants). The method yields for general marginals a 'rank-based' measure. For continuous marginals its population version coincides with the method which we propose, but for marginals with discrete components our method will have the key advantage that also in this case the marginals are always uniformly distributed. In general, the ranks of the transformed samples in our approach would coincide with classical ranks obtained using randomized tie-breaking.

For discrete distributions (in particular if the distribution is concentrated only on a few values) a transformation to the uniform distribution will make the dependence less pronounced. Thus one can not expect that the above approach will always improve the dependence detection. Nevertheless, the examples in Section \ref{sec:ex} show several cases where an improvement occurs.

\enlargethispage{\baselineskip}
In the next section we recall the distributional transform and provide a general framework to transform any multivariate dependence measure into a measure which only depends on the copula. The framework extends the work of \cite{Nes2007} and \cite{Rues2009} by focusing on the distributional properties of the corresponding sample versions. Thereafter the copula version of distance multivariance (Section \ref{sec:copm}) and dHSIC (Section \ref{sec:dhsic}) are defined and analysed. Examples are provided in Section \ref{sec:ex} and technical proofs are collected in Section \ref{sec:proofs}. In Section \ref{sec:summary} a short summary and outlook is given. In a supplement\footnote{pages \pageref{sec:supp} ff.\ of this manuscript} further simulations are provided, extending the discussions and parameter settings of the main examples.

\section{Construction of invariant dependence measures} \label{sec:disttrans}
 The following approach can be used to get invariance with respect to strictly increasing transformations for any dependence measure (actually for any function which maps random variables to some value). 
The key tool will be the so called \textit{distributional transform}, see e.g.\ \cite{Rues2009,Nes2007} and the references given therein. This 
transforms any random variable to a uniformly distributed random variable while preserving (in)dependencies between variables. For marginals with continuous distributions the population version of this method becomes the classical transformation using solely the marginal distribution function. 

The distributional transform is a one-dimensional concept. In a multivariate setting it is applied to each one-dimensional component separately. To avoid difficulties by overloading notations we discuss the one-dimensional setting first and the multivariate setting thereafter. The latter fixes the notation for the rest of the paper.

\subsection{The (univariate) distributional transform}\label{sec:disttransuni}

For a univariate random variable $X$ and $x\in \R, u\in [0,1]$ define
\begin{equation} \label{eq:def-dist-trans}
T_X(x,u):= \Prob(X< x) + u\Prob(X=x)
\end{equation} 
and let $U$ be an independent uniformly distributed random variable.
Then the distributional transform of $X$ is the random variable $T_X(X,U)$ and it has the following properties.

\begin{theorem} \label{thm:dist-trans}
With the above notations and $F_X(x):= \Prob(X\leq x)$:
\begin{enumerate}
\item $T_X(X,U)$ is uniformly distributed.
\item $X = F_X^{-1}(T_X(X,U))$ almost surely.
\item The distributional transform is invariant with respect to strictly increasing transformations, in particular translations and scalings, i.e., for any $a>0$ and $b \in \R$:
\begin{equation}
T_{aX+b}(aX+b,U) = T_X(X,U).
\end{equation}
\item \label{thm:dist-trans:dep} Univariate random variables $X_1,\ldots,X_n$ and their distributional transforms $T_1(X_1,U_1),\ldots,T_n(X_n,U_n) $ (with $U_i$ being independent and uniformly distributed) have the same copula.\footnote{Univariate random variables $X_1,\ldots,X_n$ \textit{have the copula} $C$ if $F_{(X_1,\ldots,X_n)}(x_1,\ldots,x_n) = C(F_{X_1}(x_1),\ldots,F_{X_n}(x_n))$ for all $x_i \in \R$; the independence copula is $\Pi(x_1,\ldots,x_n)= \prod_{i=1}^n x_i.$} In particular, this is the independence copula if and only if the random variables are independent. (To avoid confusion, note that if the distribution of some $X_i$ has a discrete component the copula describing the dependence of $X_1,\ldots,X_n$ is not unique - but one of these is the unique copula given via the distributional transform.)
\end{enumerate}
\end{theorem} 
\begin{proof}
The first two statements are classical, e.g.\ \cite[Lemma 3]{Nes2007} and \cite[Proposition 2.1]{Rues2009}. The third statement is a direct consequence of the following elementary identity for any strictly increasing transformation $g$ (see also \cite{SchwWolf1981}):
\begin{equation}
\begin{split}
T_{g(X)}(g(x),u)&= \Prob(g(X) < g(x)) + u \Prob(g(X) = g(x))\\
&= \Prob(X < x) + u \Prob(X = x) \phantom{mmmmmmm} = T_X(x,u).
\end{split}
\end{equation}
The last statement is the multivariate formulation of \cite[Proposition 4]{Nes2007}, see also \cite[Theorem 2.2]{Rues2009} and its proof.
\end{proof}

\begin{remark}
The copula corresponding to the random variables transformed by the distributional transform, i.e., the distribution function of $(T_1(X_1,U_1),\ldots, T_n(X_n,U_n)),$ is for discontinuous marginals also known as checkerboard copula \cite[Definition 1]{GeneNesRemiMurp2019} or multilinear extension copula. 
\end{remark}

To get the corresponding sample version one could follow the approach of \cite{GeneNesRemi2013} where especially the variable $U$ is replaced by its expectation $1/2.$ In slight variation to this we will finally replace $U$ by random samples of the uniform distribution, i.e., we use an empirical Monte Carlo version of the distributional transform. By this the limiting marginals are always uniformly distributed. 

For $x\in \R$, $u\in [0,1]$ the univariate empirical distributional transform based on the sample sequence $x^{(1)},\ldots,x^{(N)} \in \R$ (denoted by $x^{[N]}$) is
\begin{equation}
\hN T(x,u; x^{[N]}) = 
\hN T(x,u;x^{(1)},\ldots, x^{(N)}) := \frac{1}{N} \sum_{k=1}^N \left[ \One_{(-\infty,x)}(x^{(k)}) + u \One_{\{x\}}(x^{(k)}) \right],
\end{equation}
and it has the following approximation property. 
\begin{theorem}[Fundamental theorem of the empirical distributional transform] \label{thm:et-funda}
Let $X^{(k)}, k\in\N$ be independent copies of a random variable $X$, then
\begin{equation} \label{eq:funda-et}
\lim_{N\to\infty} \sup_{\substack{x\in\R \\ u\in[0,1]}} \left|\hN T(x,u;X^{(1)},\ldots, X^{(N)}) - T_X(x,u)\right| = 0 \text{ almost surely. }
\end{equation}  
\end{theorem}
\begin{proof}
The statement is a direct consequence of the fundamental theorem of statistics, the Gliwenko-Cantelli theorem (e.g.\ \cite[Thm.\ (7.4)]{Durr2005}), which states the almost sure uniform convergence of the empirical distribution function $\hN F(.):=\frac{1}{N} \sum_{k=1}^N \One_{(-\infty,.]}(X^{(k)})$ to the distribution function $F_X(.)$. Hence also $P(X<x)=F_X(x-)$ and $P(X=x)=F_X(x)-F_X(x-)$ are uniformly approximated for all $x\in\R$. Moreover the extra parameter $u$ is restricted to a bounded set (on which the functions are uniformly continuous), thus \eqref{eq:funda-et} holds.
\end{proof}

\subsection{The distributional transform in a multivariate setting} \label{sec:disttransmulti}

In the following let $\X = (X_1,\ldots,X_n)$ where each $X_i$ is a (possibly multivariate) random variable with values in $\R^{d_i}$, i.e., $X_i = (X_{i,1},\ldots,X_{i,d_i})$ and each $X_{i,k}$ is univariate. Moreover, independent copies of $\X$ are denoted by $\X^{(1)}, \ldots, \X^{(N)}$ and the latter sequence is denoted by $\vsX.$ 
The same notation is used for $\U$ (assuming additionally that all $U_{i,k}$ are independent uniformly distributed, and that these are also independent of all $X_{i,k}$). The capital letters $X$ and $U$ are replaced by $x$ and $u$, respectively, to denote corresponding samples. In this setting the distributional transform is just the vector of the distributional transforms of the components, i.e.\ using the univariate distributional transform defined via \eqref{eq:def-dist-trans} we set 
\begin{equation} 
\begin{split}
\label{eq:notation-T}
&T_{\X}(\vx,\vu) := (T_{X_1}(x_1,u_1),\ldots,T_{X_n}(x_n,u_n))\\
 \text{ where }
&T_{X_i}(x_i,u_i) := (T_{X_{i,1}}(x_{i,1},u_{i,1}),\ldots,T_{X_{i,d_i}}(x_{i,d_i},u_{i,d_i}))
\end{split}
\end{equation} 
and thus the (multivariate) distributional transform of $\X$ is $T_{\X}(\X,\U)$. By considering subgroups of the variables jointly, i.e., those represented by $X_1,\ldots,X_n$, the statement \ref{thm:dist-trans:dep} of Theorem \ref{thm:et-funda} translates directly to the multivariate setting. Thus the (in)dependencies between the components are preserved by the transform, especially: 
\begin{corollary} \label{cor:indep}
The random variables $X_1,\ldots,X_n$ are independent if and only if $T_{X_1}(X_1,U_1),\ldots,T_{X_n}(X_n,U_n)$ are independent.
\end{corollary}
The empirical distributional transform of $\vx$ given $\vu$ and $\vsx$ is given by
\begin{align}
&\hN T(\vx,\vu;\vsx) := (\hN T(x_{1},u_{1};x^{[N]}_{1}),\ldots,\hN T(x_{n},u_{n};x_n^{[N]}))\\
\text{ where }
&\hN T(x_i,u_i;x_i^{[N]})
:= (\hN T(x_{i,1},u_{i,1};x^{[N]}_{i,1}),\ldots,\hN T(x_{i,d_i},u_{i,d_i};x^{[N]})).
\end{align}
Hence, the empirical distributional transform of the whole sample sequence $\vsx$ is
\begin{equation}
\hN T(\vsx,\vsu)
:= (\hN T(\vx^{(1)},\vu^{(1)};\vx^{[N]}),\ldots,\hN T(\vx^{(N)},\vu^{(N)};\vx^{[N]})).
\end{equation}
Note that by Theorem \ref{thm:et-funda} (applied to each component) $\hN T(\vsx,\vsu)$ approximates 
\begin{equation}
T_{\X}(\vsx,\vsu) 
:=(T_{\X}(\vx^{(1)},\vu^{(1)}) , \ldots , T_{\X}(\vx^{(N)},\vu^{(N)})).
\end{equation}
The latter is the explicit distributional transform applied to the samples. It is in practical applications usually unknown (since the true marginal distributions of the $X_{i,k}$ are unknown), but for theoretic considerations it will be essential.

\subsection{The distributional transform and multivariate dependence measures}
 
The preservation of (in)dependence by the distributional transform (Corollary \ref{cor:indep} and Theorem \ref{thm:dist-trans}) is the key to transform any dependence measure $d$ to a dependence measure which only depends on the copula. Hereto let $d:\R^{d_1}\times \cdots \times \R^{d_n} \to \R$, then its copula version is defined by (using the notation of Section \ref{sec:disttransmulti}) 
\begin{equation}\label{eq:dcop}
d_{cop}(X_1,\ldots,X_n):=d(T_{X_1}(X_1,U_1),\ldots,T_{X_n}(X_n,U_n)) = d(T_{\X}(\X,\U)).
\end{equation} 
Furthermore, suppose for the measure $d$ exists an estimator $\hN d$ then a natural candidate for an estimator for $d_{cop}$ is (using the notation of Section \ref{sec:disttransmulti}) 
\begin{equation}
\hN d_{cop}(\vsx,\vsu) := \hN d (\hN T(\vsx,\vsu)).
\end{equation}  
The estimator $\hN d (\hN T(\vsx,\vsu))$ is related by \eqref{eq:funda-et} to the theoretic estimator given by  $\hN d (T_{\X}(\vsx,\vsu))$. 
 The latter is important since it fits directly into any distribution- and limit theory for $\hN d$, hereto note that the $T_{\bm{X}}(\vx^{(k)},\vu^{(k)})$ are just samples of independent copies of  $(T_{X_1}(X_1,U_1),\ldots, T_{X_n}(X_n,U_n)).$
Using this relation the next theorem states that $\hN d_{cop}$ inherits the convergence properties of $\hN d$ if the latter is uniformly continuous. But note that this does in general not hold for scaled versions of these estimators, as we will discuss afterwards. 

\begin{theorem}[preservation of consistency]\label{thm:consi}
Let $\hN d$ be a consistent estimator for $d$ which is uniformly continuous on $[0,1]^{(d_1+\ldots+d_n)N}$ in the following sense: for all $\varepsilon>0$ exists $\delta_\varepsilon >0$ such that for all $N\in \N$
\begin{equation} \label{eq:unif-cont} 
\max_{1\leq k\leq N} |\vx^{(k)}-\vy^{(k)}|<\delta_\varepsilon \quad \Rightarrow \quad |\hN d(\vx^{(1)},\ldots,\vx^{(N)}) - \hN d(\vy^{(1)},\ldots,\vy^{(N)})|<\varepsilon,
\end{equation}
and let $\X^{(1)},\ldots,\X^{(N)}$ be independent copies of $\bm{X}$.
Then $\hN d(\hN T(\vsX,\vsU))$ and $\hN d(T_{\bm{X}}(\vsX,\vsU))$ converge to the same limit in the same mode of convergence.
\end{theorem}
\begin{proof}
Let $\varepsilon>0$ and set $\delta_\varepsilon$ as in \eqref{eq:unif-cont}. Then by Theorem \ref{thm:et-funda} we can find an $N_0$ such that $|\hN T(x,u;X_{i,k}^{(1)},\ldots,X_{i,k}^{(N)}) - T_{X_{i,k}}(x,u)| < \delta_\varepsilon$ almost surely for $N > N_0$ and uniformly in $x$ and $u$. Then by \eqref{eq:unif-cont}
\begin{equation} \label{eq:dNT-dT}
|\hN d(\hN T(\vsX,\vsU)) - \hN d(T_{\bm{X}}(\vsX,\vsU))| <\varepsilon.
\end{equation}
Thus the left hand side of \eqref{eq:dNT-dT} converges to 0 almost surely as $N \to \infty$. Finally, directly or by the general implication $Y_n \xrightarrow{\Prob} Y, Z_n \xrightarrow{\Prob} 0 \ \Rightarrow \  Y_n+Z_n \xrightarrow{\Prob} Y$ the convergence of 
\begin{equation}
\hN d(\hN T(\ldots)) = \hN d(T_{\bm{X}}(\ldots)) + \bigg(\hN d(\hN T(\ldots)) - \hN d(T_{\bm{X}}(\ldots))\bigg) 
\end{equation}
is inherited from $ \hN d(T_{\bm{X}}(\vsX,\vsU))$. 
\end{proof}

A direct consequence of Theorem \ref{thm:consi} is the convergence for estimators with representations of $V$-statistic type.

\begin{corollary}[preservation of consistency for estimators of $V$-statistic type]
If 
\begin{equation} \label{eq:vstatrep}
\hN d (\vx^{(1)},\ldots,\vx^{(N)}) = \frac{1}{N^m} \sum_{i_1,\ldots,i_m = 1}^N g(x_{i_1},\ldots,x_{i_m})
\end{equation}
for some $g$ which is continuous on $[0,1]^m$, then $\hN d(\hN T(\ldots))$ and $\hN d(T_{\bm{X}}(\ldots))$ converge to the same limit in the same mode of convergence.
\end{corollary}

To test independence the consistency of $\hN d$ is usually not sufficient, instead one requires the convergence in distribution of some related statistic under $H_0.$ In the setting of $\hN d$ being a V-statistic the corresponding test statistic is of the form $N^{\beta}\cdot \hN d$ for some $\beta>0$. Theoretically here again a uniformity of the convergence is required to transfer results from $\hN d (T_{\bm{X}}(\ldots))$ to $\hN d (\hN T(\ldots))$. But, as the following counterexample shows, for the proof of such a convergence it seems necessary to go into the details of the specific statistic. 

\begin{remark}[Counterexample - for scaled V-statistics the distributional limit is in general not preserved when replacing $T_X$ by $\hN T$]
We give an elementary example (without explicit link to dependence measures). Let $X^{(i)}$ be independent copies of a continuous random variable $X$. Due to the continuity we have: $T_X = F_X.$ Define $g(x):= 12 (x-\frac{1}{2})$ then by the strong law of large numbers $\hN V(F_X(X^{(1)}),\ldots,F_X(X^{(N)})) :=\frac{1}{N} \sum_{i=1}^N g(F_X(X^{(i)})) \xrightarrow{a.s.} 0$ since $F_X(X^{(i)})$ are uniformly distributed with mean $\frac{1}{2}$ and variance $\frac{1}{12}$. Moreover, by the central limit theorem 
$\sqrt{N} \cdot \hN V(F_X(X^{(1)}),\ldots,F_X(X^{(N)})) \xrightarrow{d} Z \sim N(0,1)$. But in this case the distance to the V-statistic with $F_X$ replaced by $\hN F$ does not vanish, in fact:
\begin{equation*}
\sqrt{N}\cdot (\hN V(\hN F(X^{(1)}),\ldots,\hN F(X^{(N)})) - \hN V(F_X(X^{(1)}),\ldots,F_X(X^{(N)}))) \xrightarrow{d} Z' \nequiv 0
\end{equation*}
where $\hN F(.)$ denotes the empirical distribution function of $X^{(1)},\ldots,X^{(N)}.$ The above convergence is a consequence of the fact that by the central limit theorem the term $\sqrt{N} (\hN F(x) - F_X(x))$ converges in distribution to $Z'' \sim N(0,\Prob(X\leq x)\Prob(X>x)).$

This shows that in general the limits of $\sqrt{N} \hN V(F_X(\ldots))$ and $\sqrt{N} \hN V(\hN F(\ldots))$ differ.
\end{remark}

We aim to use the explicitly known limit behaviour of the underlying estimators without the requirement to reprove these. In this case the limit distribution of the estimator only depends on the copula and on the uniform marginals. This is (besides the margin free quantification of dependence) a further benefit for practical applications of the presented approach, since it can yield computationally faster test procedures.

\begin{remark}[Speed advantage of methods based on the distributional transform] \label{rem:speed}
The knowledge that the limit margins are uniformly distributed (by Theorem \ref{thm:dist-trans}), allows to precompute required quantities, which otherwise (without the distributional transform) would require additional information prior to the test.

A basic example is a Monte Carlo p-value derivation: 
Suppose one knows (as it will be the case in our setting) that the distribution of the limit of the test statistic does not depend on the marginals, and one wants to perform many tests in the setting of $n$ univariate marginals with a sample size of $N$. Then one can easily obtain an approximate distribution function of the given estimator under $H_0$ (i.e., for independent marginals) based on a large number of Monte Carlo samples (with each sample being of size $N$) of $n$ independent uniformly distributed random variables. This distribution function can then be reused in every test. The quality of this approximation turns out to be good in our setting (see Figure \ref{fig:MCU-vs-MC}). Formally this means (in our case) that the distribution of the {\bf approximate Monte Carlo $\bm{H_0}$ samples} $\sqrt{N}\cdot \hN d(\vu^{H_0,(1)},\ldots,\vu^{H_0,(N)} )$ is close to the distribution of {\bf exact Monte Carlo $\bm{H_0}$ samples} $\sqrt{N} \cdot \hN d(\hN T(\vx^{H_0,(1)},\ldots, \vx^{H_0,(N)},\vsu))$, where $\vu^{H_0,(k)}$ are vectors with samples of independent uniformly distributed random variables and $\vx^{H_0,(k)}$ are samples of $(X_1,\ldots,X_n)$ with independent components.

Note, that in general this method is not applicable if the variables under consideration are multivariate, since in this case the marginal distributions under $H_0$ are multivariate. Each margin then still consists of univariate uniformly distributed components, but these components can be dependent.
\end{remark}
\enlargethispage{2\baselineskip}

This section presented a general approach for the construction of dependence measures based on the copula. As stated before, via the $u_i^{(k)}$ some further randomness is introduced and this might or might not blur a given dependence, see the examples in Section \ref{sec:ex}.

\section{Copula distance multivariance}\label{sec:copm}
Distance multivariance is defined by (cf.\ \cite{Boet2020})
\begin{equation}
M(X_1,\ldots,X_n):= \sqrt{\int \left|\E\left(\prod^{n}_{i=1} (\ee^{\ii X_i\cdot t_i} - f_{X_i}(t_i))\right)\right|^2\,\rho(dt)},
\end{equation}
where $X_i$ are $\R^{d_i}$ valued random variables with characteristic functions $f_{X_i}(t_i):= \E(\ee^{\ii X_i \cdot t_i})$ and $\rho = \otimes_{i=1}^n \rho_i$ is based on symmetric  measures $\rho_i$ with full support on $\R^{d_i}$ such that $\int 1\land |t_i|^2 \, \rho_i(dt_i) < \infty$. For the measures $\rho_i$ there are many choices (cf.\ \cite[Table 1]{BoetKellSchi2018}) which unify several dependence measures in the case of $n = 2$ (see \cite[Section 3]{Boet2020}) and extend these to a multivariate setting. There is a one-to-one correspondence between each measure $\rho_i$ (on $\R^{d_i}\backslash\{0\}$) and the real valued continuous negative definite function $\psi_i$ given by
\begin{equation}\label{eq:cndf}
\psi_i(x_i):= \int_{\R^{d_i}\backslash\{0\}} 1-\cos(x_i\cdot t_i)\,\rho_i(t_i).
\end{equation}

Based on the discussion in the previous section and using the notation of \eqref{eq:notation-T} we define the copula version of distance multivariance for $X_1,\ldots,X_n$ by
\begin{align}
\notag M_{cop}(X_1,\ldots,X_n) :&= M(T_{X_1}(X_1,U_1),\ldots,T_{X_n}(X_n,U_n)) \\&=  M(T_{X_i}(X_i,U_i),i\in \{1,\ldots,n\})
\end{align}
and analogously the copula version of total distance multivariance $\overline{M}$ is given by
\begin{equation}
\begin{split}
\overline{M_{cop}}(X_1,\ldots,X_n) &:= \overline{M}(T_{X_1}(X_1,U_1),\ldots,T_{X_n}(X_n,U_n)) \\
&:= \sqrt{ \sum_{\substack{I\subset \{1,\ldots,n\}\\ |I|>1}} M^2(T_{X_i}(X_i,U_i),i\in I) }.
\end{split}
\end{equation}
In \cite{Boet2020} further measures based on $M$ are discussed, e.g.\ the normalized multivariance $\Mskript$ and the $m$-multivariances $M_m$ and $\Mskript_m$, also for these the corresponding copula versions can be defined analogously to the above.

The key observation for the new measures is that these inherit the following properties.

\begin{theorem}[Characterization of independence] \label{thm:Mcop-indep}
\begin{equation}
\overline{M}_{cop}(X_1,\ldots,X_n) = 0 \text{ if and only if $X_1,\ldots,X_n$ are independent.}
\end{equation}
In particular, for random variables $X_1,\ldots, X_n$ which are $(n-1)$-independent
\begin{equation}
M_{cop}(X_1,\ldots,X_n) = 0 \text{ if and only if $X_1,\ldots,X_n$ are independent.}
\end{equation}
\end{theorem}
\begin{proof}
By Corollary \ref{cor:indep} the random variables $X_1,\ldots,X_n$ are independent if and only if $T_{X_1}(X_1,U_1),\ldots,T_{X_n}(X_n,U_n)$ are independent and the same equivalence holds for any subfamily. Thus the results are a direct consequence of the corresponding properties of distance multivariance \cite[Theorem 2.1]{Boet2020}. 
\end{proof}

For samples $\vx^{(k)} = (x_1^{(k)},\ldots,x_n^{(k)}), k=1,\ldots,N$ the sample version of distance multivariance is given by
\begin{equation}
\hN M (\vx^{(1)},\ldots,\vx^{(N)}) := M( \hat X_i,\ldots,\hat X_n)
\end{equation}
where each random variable $\hat X_i$ is distributed according to the empirical distribution of $x_i^{(1)},\ldots,x_i^{(N)}.$ Sample distance multivariance has also an alternative representation given in \eqref{eq:Msample} and it can be turned into a numerical efficient estimator using distance matrices, for details we refer to \cite{BoetKellSchi2019,Boet2020}.

We define (using the notation of Section \ref{sec:disttransmulti}) 
\begin{equation} \label{eq:NMcop}
\hN M_{cop}(\vx^{(1)},\ldots,\vx^{(N)},\vsu) := \hN M(\hN T(\vsx,\vsu)).
\end{equation}

The following result provides everything required for corresponding independence tests based on asymptotics. The technical proof is postponed to Section \ref{sec:proofs}. 

\begin{theorem}[Asymptotics of $\hN M_{cop}$] \label{thm:Mcop}
$\hN M_{cop}$ and $N \cdot \hN M_{cop}^2$ inherit the distributional properties of $\hN M$ and $N \cdot \hN M^2$, respectively. 

In particular this yields for all random variables $\X',\X^{(k)}, k\in \N$ which are independent copies of $\X=(X_1,\ldots,X_n)$ (without any moment assumptions):
\begin{align}
\label{eq:Mcop-consistent}\hN M_{cop}(\X^{(1)},\ldots,\X^{(N)},\vsU) &\xrightarrow{a.s.}  M_{cop}(X_1,\ldots,X_n),\\
\label{eq:Mcop-divergence}N\cdot\hN M_{cop}^2(\X^{(1)},\ldots,\X^{(N)},\vsU) &\xrightarrow{a.s.} \infty \quad \text{ if $X_1,\ldots,X_n$ are dependent,}\\
\label{eq:Mcop-convergence}N\cdot\hN M_{cop}^2(\X^{(1)},\ldots,\X^{(N)},\vsU) &\xrightarrow{d} \mathbb{G} \quad \quad \text{ if $X_1,\ldots,X_n$ are independent,}
\end{align}
where $\mathbb{G}$ is a Gaussian quadratic form with expectation given by $\E(\mathbb{G}) = \prod_{i=1}^n \E(\psi_i( T_{X_i}(X_i,U_i)-T_{X_i}(X'_i,U_i')  )).$ 
\end{theorem}
Analogous results to \eqref{eq:Mcop-consistent}-\eqref{eq:Mcop-convergence} hold also for the normalized multivariance and $m$-multivariances.

Technically, when using $\hN M_{cop}$, the methods in \cite{Boet2020} are only preceded by the empirical distributional transform. Therefore we omit here a further description of the tests and refer to the extended expositions in \cite{Boet2020} and \cite{BersBoet2018v2}. 

\begin{remark}[Speed advantage - using precalculated parameters]\label{rem:pearson}
Due to the known uniform marginals direct p-value estimates are possible in the case of univariate marginals using the methods described in \cite[Example 5.6]{BersBoet2018v2}. The required values are: limit mean 1/3, limit variance 2/45 and limit skewness 8/945, moreover (not given in \cite{BersBoet2018v2}) the parameters required for the finite sample estimates \cite[Theorems 4.15, 4.17]{BersBoet2018v2} become $b=1/6$, $c=7/60$ and $d=1/9.$  These known values provide a considerable speed gain in comparison to the moment estimation methods of \cite{BersBoet2018v2}, see Figure \ref{fig:H0-speed}. 
For multivariate marginals these parameters cannot be precomputed in general, since the required values depend on the (typically) unknown dependence of the components within the marginals (cf.\ Remark \ref{rem:speed}).
\end{remark}

\section{Copula dHSIC} \label{sec:dhsic}

With the notation of the previous section (see also \cite[Secion 3.3]{Boet2020}) the multivariate Hilbert-Schmidt-Independence-Criterion of \cite{PfisBuehSchoPete2017} for real valued random vectors is given by
\begin{align}
\notag dHSIC(X_1,\ldots,X_n) = &\E\left[ \prod_{i=1}^n (1-\psi_i(X_i-X_i')) \right]- 2 \E\left[ \prod_{i=1}^n \E(1-\psi_i(X_i-X_i')|X_i)\right]\\
& + \prod_{i=1}^n \E(1-\psi_i(X_i-X_i')),
\end{align}
where $(X_1',\dots,X_n')$ is an independent copy of $(X_1,\dots,X_n)$ and each $\psi_i$ is bounded with $\psi_i(0)=1$. Note that this implies (since $\psi_i$ itself is a continuous negative definite function, see \eqref{eq:cndf}), that
\begin{equation}
k_i(x_i,x_i'):=\overline{\psi_i}(x_i-x_i'):= 1-\psi_i(x_i-x_i')
\end{equation}
is a positive definite kernel, cf.\ \cite[Section 3.2]{Boet2020}. 
Then the copula version of $dHSIC$ is given by
\begin{equation}
\begin{split}
dHSIC_{cop}(X_1,\ldots,X_n) :&= dHSIC(T_{X_1}(X_1,U_1),\ldots,T_{X_n}(X_n,U_n)) \\&=  dHSIC(T_{X_i}(X_i,U_i),i\in \{1,\ldots,n\})
\end{split}
\end{equation}
and analogous to Theorem \ref{thm:Mcop-indep} (using \cite[Propostion 1]{PfisBuehSchoPete2017}) the measure $dHSIC_{cop}$ characterizes independence.
\begin{theorem}[Characterization of independence]
\begin{equation}
dHSIC_{cop}(X_1,\ldots,X_n) = 0  \text{ if and only if $X_1,\ldots,X_n$ are independent.}
\end{equation}
\end{theorem}
An empirical estimator for $dHSIC$ is given by
\begin{align}
\notag \hN dHSIC(\vx^{(1)},\ldots,\vx^{(n)}) := &\frac{1}{N^2} \sum_{j,k=1}^N \prod_{i=1}^n \overline{\psi}_i(x_i^{(j)}-x_i^{(k)})
+ \prod_{i=1}^n \left(\frac{1}{N^2} \sum_{j,k=1}^N \overline{\psi}_i(x_i^{(j)}-x_i^{(k)}) \right)\\
\label{eq:def-Ndhsic}& - 2 \frac{1}{N} \sum_{j=1}^N \prod_{i=1}^n \left( \frac{1}{N} \sum_{k=1}^N \overline{\psi}_i(x_i^{(j)}-x_i^{(k)}) \right).
\end{align}
Note that this might look slightly different to the estimator defined in \cite[Definition 4]{PfisBuehSchoPete2017}, but an expansion of the products and a relabelling of the indices yields their representation, see \eqref{eq:Ndhsic-expanded}. Using \eqref{eq:def-Ndhsic} the estimator can be defined for all $N,$ whereas $N\geq 2n$ was required in \cite{PfisBuehSchoPete2017}.

Analogous to \eqref{eq:NMcop} the estimator for $dHSIC_{cop}$ is defined (using the notation of Section \ref{sec:disttransmulti}) by
\begin{equation} \label{eq:Ndhsiccop}
\hN dHSIC_{cop}(\vx^{(1)},\ldots,\vx^{(N)},\vsu) := \hN dHSIC(\hN T(\vsx,\vsu)). 
\end{equation}
For this estimator the following theorem holds, details of the proof are in Section \ref{app:dhsic}.

\begin{theorem}[Asymptotics of $\hN dHSIC_{cop}$] \label{thm:dhsiccop}
$\hN dHSIC_{cop}$ and $N \cdot \hN dHSIC_{cop}$ inherit the distributional properties of $\hN dHSIC$ and $N \cdot \hN dHSIC$, respectively. 
\end{theorem}

Tests based on $\hN dHSIC$ use either a resampling method, a rough gamma approximation or an eigenvalue method (see in particular \cite[Table 1]{PfisBuehSchoPete2017} and Figure \ref{fig:H0-speed}.

\section{Empirical properties of $M_{cop}$ and $dHSIC_{cop}$} \label{sec:ex}

We will show that the new measures can be more powerful than various other copula based measures. Thereafter we evaluate computation methods for the p-values based on conservative behaviour and speed. Moreover we will also indicate the limitations of the introduced measures by giving an example where the copula versions of the measures perform worse than the original measures. This section is complemented by several figures and tables in the supplement\footnote{pages \pageref{sec:supp} ff.\ of this manuscript}, these consider mostly the same examples with parameter variations and extend detail. 

All simulations are performed on an i7-6500U CPU Laptop using the statistical computing environment R \cite{RCT2019}, in particular with the packages \texttt{copula} \cite{HofeKojaMaecYan2018}, \texttt{dHSIC} \cite{PfisPete2019} and  \texttt{multivariance} \cite{Boett2019R-2.2.0}. 

The tests are performed with significance level 0.05 using 1000 samples. The power is denoted in percent. For multivariance the standard measures $\rho_i$ corresponding to the Euclidean distance $\psi_i(x_i)=|x_i|$ were used, and for dHSIC the kernel defined via  $\overline{\psi_i}(x)= 1- \exp(-|x|^2/(2\delta^2))$ with $\delta = 3$ was used (for different values of $\delta$ see Table \ref{tab:dhsic-variants}). 
Note that both, multivariance and dHSIC, allow many other (possibly parametrized) variants. These would certainly allow to improve the performance for particular examples, but these would also require prior to testing some knowledge of the type of dependence. The purpose of the examples in this section is to show that the copula based measures can be competitive. The task to find optimal measure selection procedures will be part of future research. Moreover, for both measures exist various p-value derivation methods which will be compared in Figure \ref{fig:H0-speed}. If not stated otherwise, we use a reference distribution based on 100000 approximate Monte Carlo $H_0$ samples as described in Remark \ref{rem:speed} to determine the p-values.

For a comparison with other copula based measures we consider a set of examples discussed in \cite{GeneNesRemiMurp2019} which provide (using their numbers) a direct comparison to eight of their dependence measures, including those introduced in \cite{GeneNesRemi2013}. Hereto samples of dependent uniformly distributed random variables are obtained using the following copulas: Clayton copula, Student copula with 1 and 3 degrees of freedom, Normal copula, Frank copula and Gumbel copula. See Figure \ref{fig:copulas} for a visualization of the induced dependencies. Each copula is parametrized such that the pairwise Kendall's tau is equal to 0.1 (simulations using other values for tau can be found in Table \ref{tab:vary-tau}). The dependent samples were transformed to the following types of marginal distributions: Poisson with mean 1 and 20 (P1 and P20), rounded Pareto (RP) with survival function $1/(k+1)^{1/3}$ for $k \in \N_0$ (discrete, with infinite expectation), Cauchy (CA) (continuous, with no expectation), Student with 3 degrees of freedom altered with an atom at 0 of mass 0.05 (SA) (mixture, with infinite variance).

The power comparison in Table \ref{tab:comp} shows that in certain cases the tests based on $\overline{M}_{cop}$ and $dHSIC_{cop}$ are more powerful than those based on measures considered in \cite{GeneNesRemiMurp2019} ('min' and 'max' denote the minimal and maximal power of the tests considered therein (without the measure 'R'); see Table \ref{tab:fullGene} for details). Moreover, in all cases the new measures provide tests which are more powerful than the minimum of the competing measures. $\overline{M}_{cop}$ performs particularly well for the Student copula (with the exception of Poisson marginals), $dHSIC_{cop}$ handles particularly well the cases of normal, Clayton and Frank copula (in all cases except P1 it is at least close to the 'max'). It is interesting to note that this preference differs from the preferences of the measures $S_n$ and $T_n$ which (as stated in the introduction) correspond structurally to dHSIC and multivariance, respectively.  

\begin{table}[h] 
\centering
\footnotesize
\begin{tabular}{llllll}
 copula & type & $\overline{\Mskript}_{cop}$ & $dH_{cop}$ & min & max \\ 
  \hline
normal & CA & 72.9 & \textbf{83.2 } & 67.8 & 82.6 \\ 
   & P1 & 53.1 & 63.1 & 49.4 & \textbf{82.5 } \\ 
   & P20 & 72.6 & \textbf{82.7 } & 66.2 & 82.4 \\ 
   & RP & 71.5 & 81.1 & 66.8 & \textbf{83.6 } \\ 
   & SA & 72.9 & \textbf{83.0 } & 68.1 & 82.5 \\ 
  t1 & CA & \textbf{100 } & 83.1 & 78.6 & 99.2 \\ 
   & P1 & 90.2 & 65.3 & 50.3 & \textbf{93.4 } \\ 
   & P20 & \textbf{100 } & 82.7 & 78.5 & 99.2 \\ 
   & RP & \textbf{100 } & 81.0 & 80.2 & 98.6 \\ 
   & SA & \textbf{100 } & 82.8 & 78.7 & 99.2 \\ 
  t3 & CA & \textbf{96.5 } & 82.6 & 66.7 & 93.4 \\ 
   & P1 & 71.0 & 67.0 & 47.9 & \textbf{81.8 } \\ 
   & P20 & \textbf{96.3 } & 82.3 & 65.7 & 92.9 \\ 
   & RP & \textbf{94.1 } & 81.9 & 64.5 & 91.9 \\ 
   & SA & \textbf{96.6 } & 82.8 & 66.6 & 93.6 \\ 
   \hline
\end{tabular}

\begin{tabular}{llllll}
 copula & type & $\overline{\Mskript}_{cop}$ & $dH_{cop}$ & min & max \\ 
  \hline
clayton & CA & 79.7 & \textbf{85.5 } & 67.9 & 83.3 \\ 
   & P1 & 48.2 & 53.4 & 42.3 & \textbf{69.5 } \\ 
   & P20 & 77.7 & \textbf{84.7 } & 66.8 & 83.9 \\ 
   & RP & 75.2 & \textbf{81.8 } & 65.0 & 81.4 \\ 
   & SA & 79.7 & \textbf{85.3 } & 67.5 & 83.5 \\ 
  frank & CA & 81.8 & 85.7 & 68.6 & \textbf{85.8 } \\ 
   & P1 & 65.3 & 65.6 & 54.1 & \textbf{79.9 } \\ 
   & P20 & 82.3 & 84.9 & 68.2 & \textbf{86.1 } \\ 
   & RP & 80.8 & 84.9 & 67.8 & \textbf{85.7 } \\ 
   & SA & 81.8 & \textbf{85.9 } & 68.4 & 85.7 \\ 
  gumbel & CA & 85.9 & 81.5 & 60.0 & \textbf{88.4 } \\ 
   & P1 & 73.9 & 65.2 & 48.5 & \textbf{84.4 } \\ 
   & P20 & 85.4 & 80.8 & 60.4 & \textbf{88.0 } \\ 
   & RP & 85.8 & 80.6 & 60.4 & \textbf{87.9 } \\ 
   & SA & 85.8 & 81.6 & 59.4 & \textbf{88.3 } \\ 
   \hline
\end{tabular}

\medskip
\caption{Comparison of the power (stated in \%) of independence tests based on  $\Mskript_{cop}$ and $dHSIC_{cop}$ with those given in \cite{GeneNesRemiMurp2019}. Setting $n=5, N=100$. The 'min' and 'max' are values of the tests of \cite[Table 5]{GeneNesRemiMurp2019} (and its extension \cite[Supplement Table S15]{GeneNesRemiMurp2019}) without the test 'R'. The maximum is always printed in bold. See Table \ref{tab:fullGene} for the full table.} 
\label{tab:comp}
\end{table}

If one compares tests based on the new measures with those using their base measures $\overline{\Mskript}$ and $dHSIC$, it also turns out that sometimes the base measures and sometimes the new measures yield more powerful tests (see e.g.\ Table \ref{tab:fullGene}). Moreover, the behaviour also varies with the sample size and the dimension (see Tables \ref{tab:vary-N} and \ref{tab:vary-n}).

Recall that using the empirical distribution of the test statistic with Monte Carlo samples of the $H_0$ distribution (i.e., samples with independent components) provides (almost) exact p-values. Furthermore, note that if the marginals are from arbitrary distributions the corresponding finite sample distribution under $H_0$ does not coincide with the distribution based on uniformly distributed marginals. Nevertheless the latter approximation performs reasonably well for all marginals of Table \ref{tab:comp}, see Figure \ref{fig:MCU-vs-MC}. This method becomes in the setting of multiple tests very efficient, cf.\ Remark \ref{rem:speed}. But if the marginal distributions are multidimensional or if $n$ and/or $N$ are often varied this Monte Carlo approach becomes very slow or inapplicable. Therefore we will now look at other p-value derivation methods. In particular, we illustrate their (non-)conservative behaviour and speed in Figure \ref{fig:H0-speed}. Here we only consider uniform marginals. The corresponding exact Monte Carlo p-values are used as benchmark and plotted against the p-values obtained via the various methods which are available for multivariance and dHSIC. Hence Figure \ref{fig:H0-speed} allows to visually assess the empirical size for various significance levels at once. For multivariance the use of the method described in Remark \ref{rem:pearson} (called 'pearson\_uniform' in the figure) is the fastest and Pearson's approximation is the sharpest. The latter is here almost indistinguishable from the benchmark, and it is in general the recommended method for fast p-value derivations for independence tests based on multivariance \cite{BersBoet2018v2}. For dHSIC the eigenvalue method turns out to be a good choice (in its current implementation it is about 10 times slower than 'pearson\_uniform' and slightly conservative). The gamma approximation, which was included in \cite{PfisBuehSchoPete2017} as a fast unproven alternative, seems in this setting not reliable (it shows very liberal behaviour).

\begin{figure}[H]
\centering
p-value derivation methods for $\overline{\Mskript}_{cop}$\\
\includegraphics[width=0.24\textwidth]{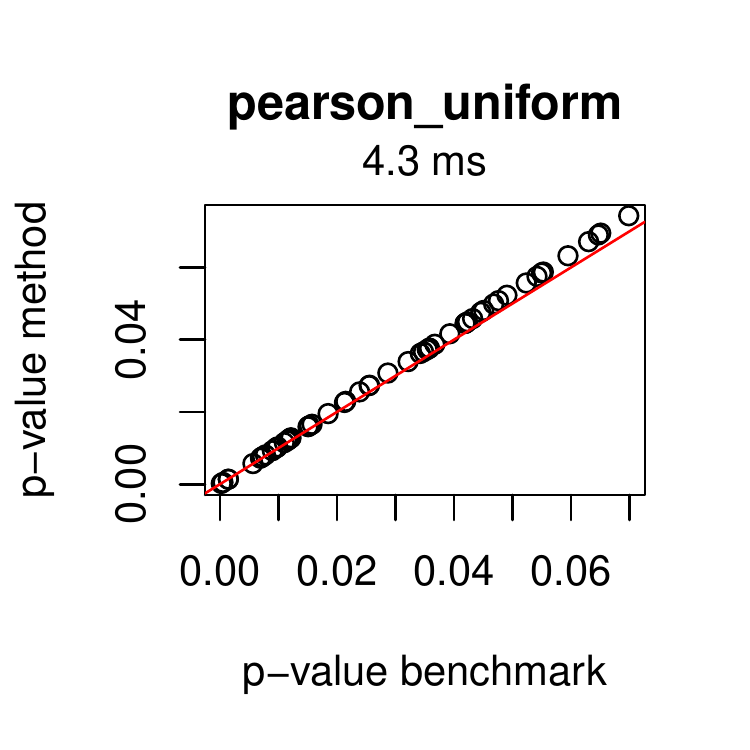} 
\includegraphics[width=0.24\textwidth]{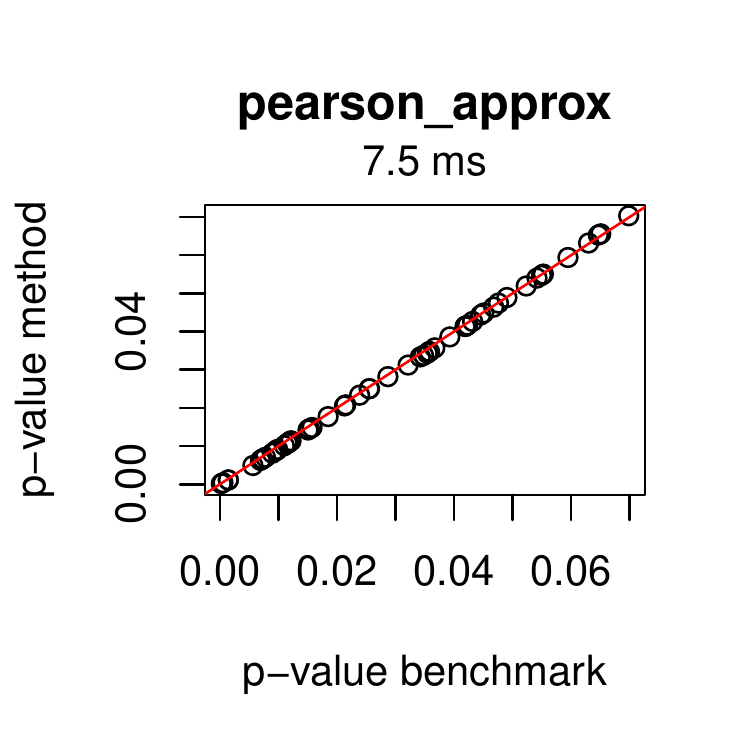} 
\includegraphics[width=0.24\textwidth]{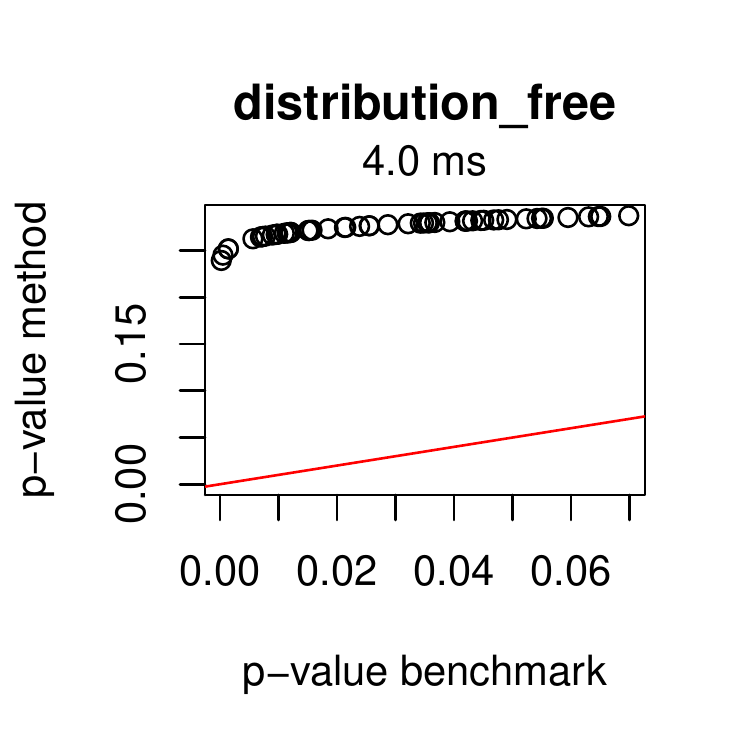} 
\includegraphics[width=0.24\textwidth]{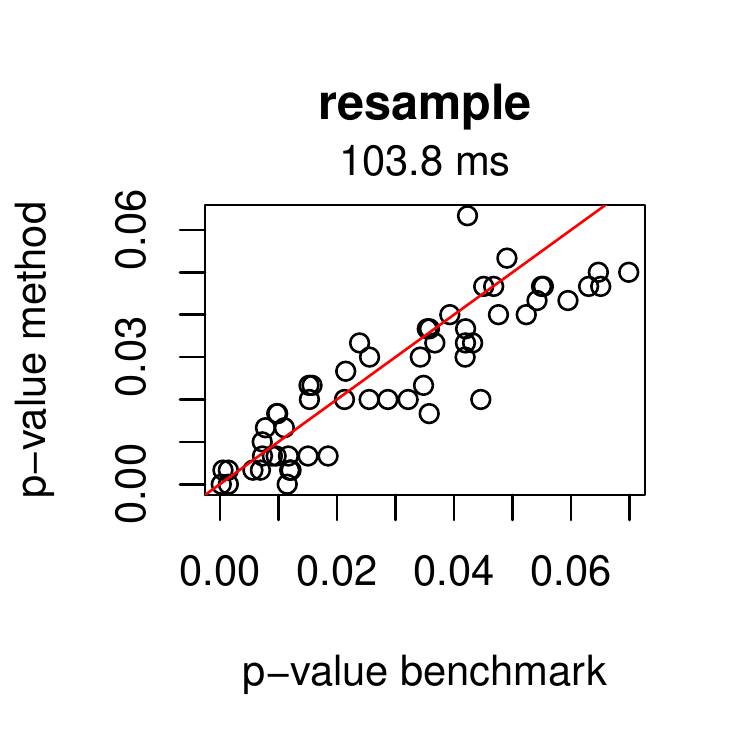} \\
p-value derivation methods for $dHSIC_{cop}$\\
\includegraphics[width=0.24\textwidth]{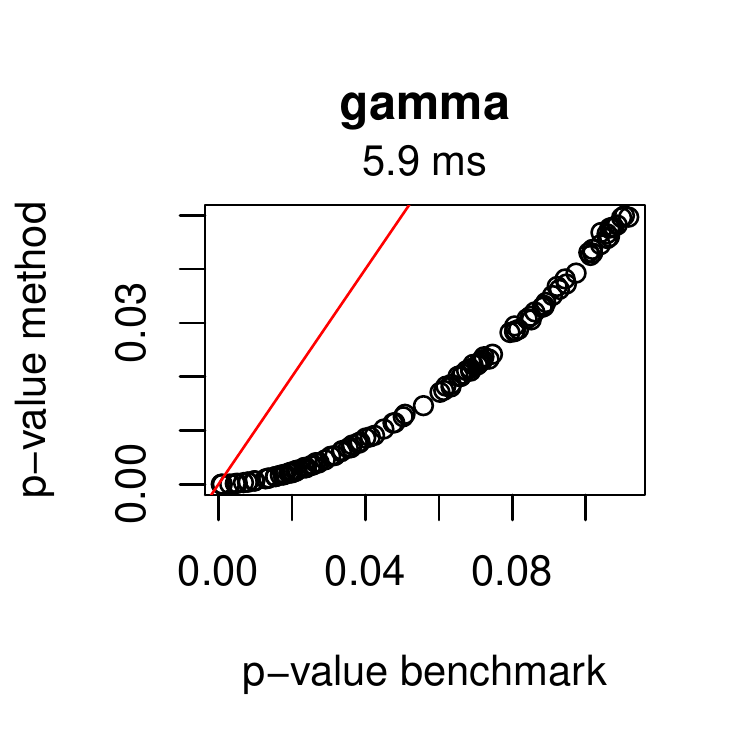} 
\includegraphics[width=0.24\textwidth]{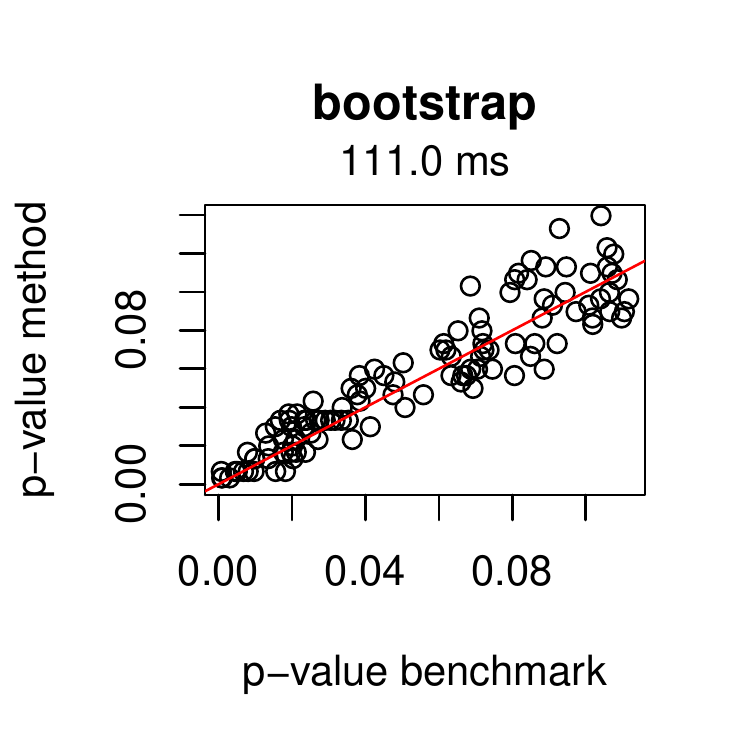} 
\includegraphics[width=0.24\textwidth]{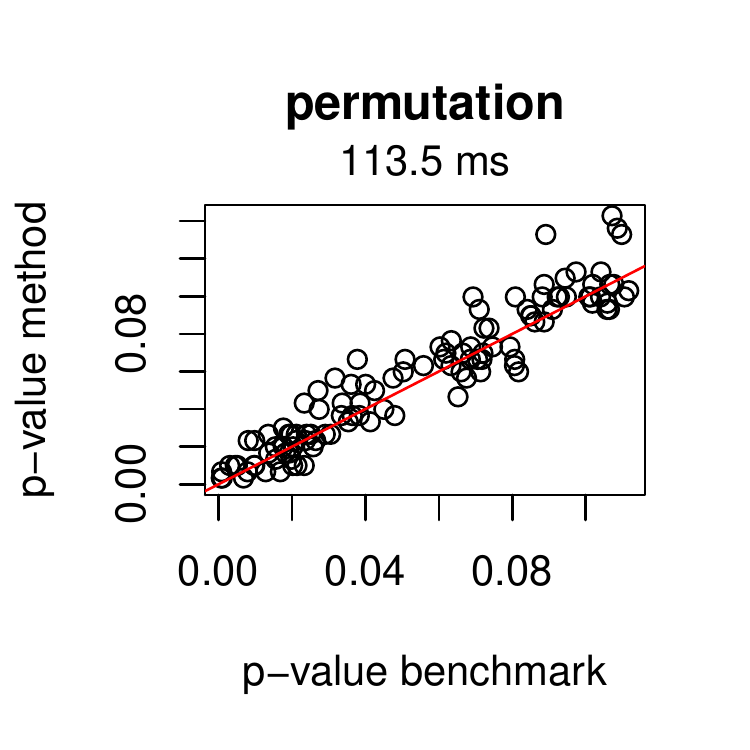} 
\includegraphics[width=0.24\textwidth]{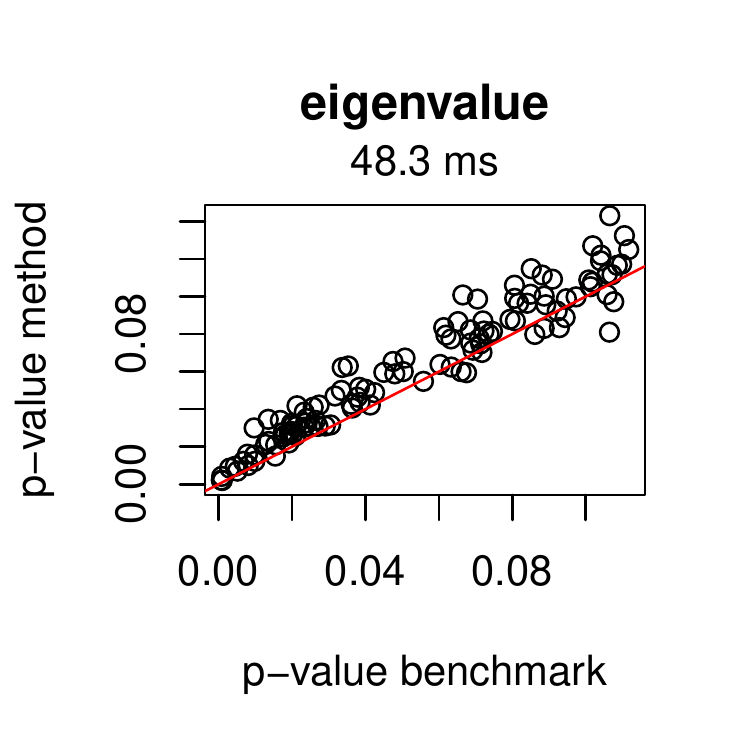} 
\caption{\normal Comparison of p-value derivation methods under $H_0$. Setting: $N=100,n=5$ with uniform marginals. Note that points above the line indicate conservative behaviour, thus in practice one prefers (besides speed considerations) methods which are closest to the line or at least not systematically below. The method 'pearson\_uniform' is based on Remark 3.3, it uses a quadratic form estimate by a Pearson distribution with precalculated moments for uniformly distributed marginals. The other methods are the options provided in the corresponding R packages [23, 24] (for multivariance: 'pearson\_approx' - quadratic form estimate based on estimated moments, 'distribution\_free' - conservative quadratic form estimate, 'resample' - resampling without replacement; for dHSIC: 'gamma' - estimate using a gamma distribution, 'bootstrap' - resampling with replacement, 'permutation' - resampling without replacement, 'eigenvalue' - estimate based on estimated eigenvalues; for the resampling methods  300 resamples were used). The stated median computation time is for a single p-value and each contains about 3.3 ms for the empirical transform. The benchmark p-value median computation time is 4.3 ms (based on 100000 precalculated exact Monte Carlo $H_0$ samples). Observation: In terms of speed and sharpness 'pearson\_approx' can be recommended for $\overline{\Mskript}_{cop}$ and 'eigenvalue' for ${dHSIC}_{cop}$. Moreover, 'pearson\_uniform'  provides for multivariance an additional speed advantage, but this is accompanied by a slightly conservative behaviour.}
\label{fig:H0-speed}
\end{figure}

In the examples of Table \ref{tab:comp} all variables are pairwise dependent. An example of pairwise independent but dependent random variables is constructed as follows (also known as Bernstein's coins, cf.\ \cite[Section 5]{BoetKellSchi2019} and \cite[Example 10.2]{Boet2020}): Let $X_1$ and $X_2$ be independent Bernoulli distributed random variables and set $X_3:= \One_{\{X_1\}}(X_2)$, which models the event that both 'coins' show the same side. Then all three variables are Bernoulli distributed and feature the dependence structure shown in Figure \ref{fig:coins} (see \cite{Boet2020} for further details on the visualization of higher order dependencies). For these random variables the detection power of tests based on $M_{cop}$ and $dHSIC_{cop}$ and of their classical counter parts (i.e., without the distributional transformation) are shown in the central plot of Figure \ref{fig:coins}. Here clearly the copula versions perform worse. If we perturb the $X_i$ by independent normally distributed random variables with variance $1/2$ and mean $0$ the measures $M_{cop}$ and $M$ perform similarly, but for $dHSIC$ a difference is still visible. Note the comparison of  $dHSIC$ and $dHSIC_{cop}$ might be considered unfair, since it depends strongly on a bandwidth parameter which was fixed here (as stated at the beginning of this section; for cases with variable bandwidth see Table \ref{tab:dhsic-variants}). 

\begin{figure}[H]
\centering 
\newlength{\myl}
\setlength{\myl}{10.5\baselineskip}

\includegraphics[height=\myl]{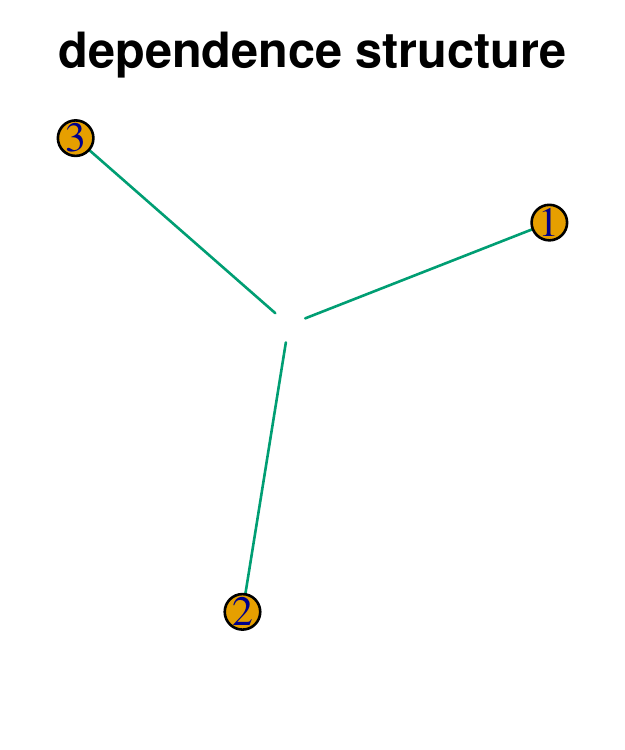} \includegraphics[height=\myl]{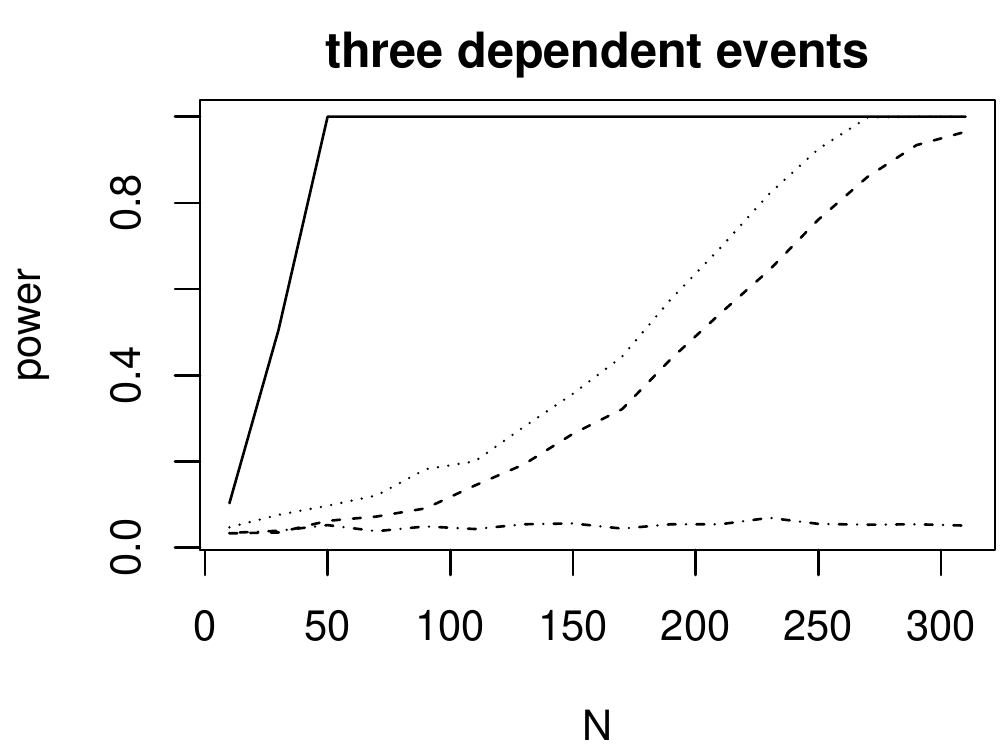} 
\includegraphics[height=\myl]{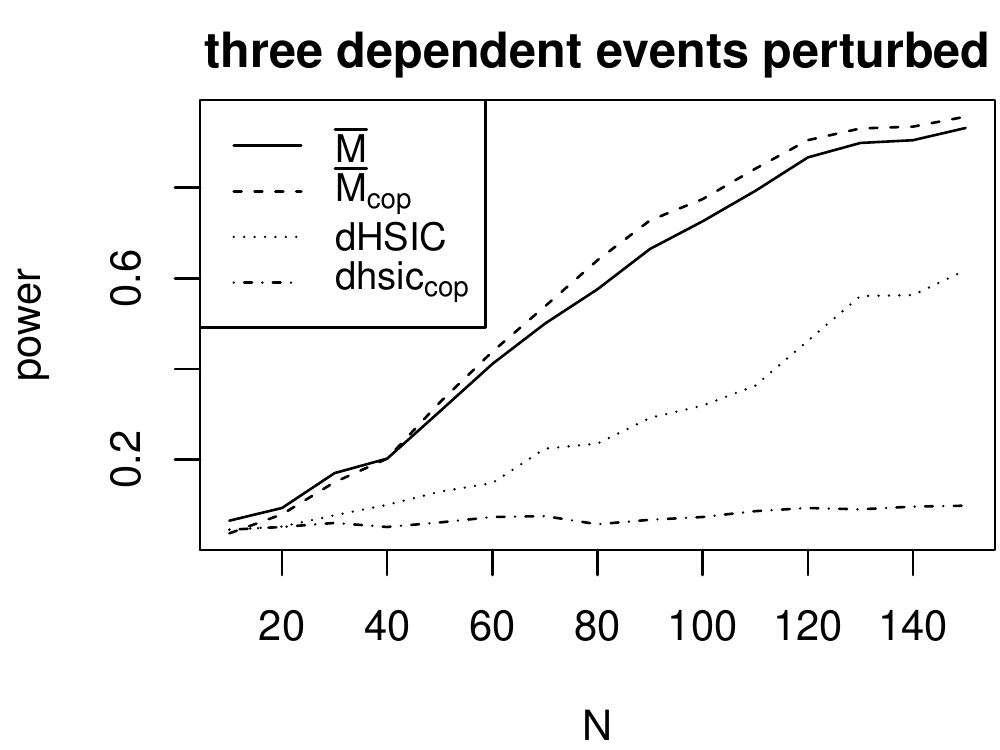} 
\caption{Dependence structure sketch and power of the independence tests for random variables modeling three pairwise independent but dependent events, with and without normal perturbation.}
\label{fig:coins}
\end{figure}

\section{Summary and outlook} \label{sec:summary}
A general scheme to transform dependence measures to measures which only depend on an underlying copula was discussed, 
 with a focus on the distributional properties of the corresponding estimators. The scheme was then explicitly applied to dHSIC and distance multivariance, yielding new measures and corresponding independence tests which are competitive to other copula based measures. 

The current work provides an essential basis for future research in several directions, e.g.:
\begin{itemize}
\item Quantification of dependence, e.g.\ how to interpret the values of the multicorrelations introduced in \cite{Boet2020} for the corresponding copula versions (including a discussion of corresponding dependence measure axioms). This might enable comparisons directly based on the values of the dependence measures. Hence selection procedures, e.g.\ as those suggested in \cite{GoreKrzyWoly2019}, become valid.
\item Optimization and classification of variants of the introduced measures, e.g.\ discussions of the selection of the underlying $\psi_i$. Hereto note that \cite[Section 5.2]{BoetKellSchi2019} provides an example of dependent random variables with uniform marginals where the performance of standard (using $\rho$ corresponding to the Euclidean distances) multivariance improved when the underlying measure $\rho$ was changed (within the framework of distance multivariance).
\item A detailed comparison of dHSIC and distance multivariance. A clear rule of thumb indicating the preference of one of the measures for a given situation is still missing. The alternating optimum in Tables \ref{tab:comp}, \ref{tab:vary-N}, \ref{tab:vary-n} and \ref{tab:dhsic-variants} indicates that (also for the copula versions) a simple answer is not to be expected, see in this context also the discussion in \cite[Section 3.3]{Boet2020}.
\end{itemize}
 
\section{Technical details}
 \label{sec:proofs}
\subsection{Proof of Theorem \ref{thm:Mcop} (Asymptotics of $\hN M_{cop}$)}
We use the notation of \cite[Section 9.3]{Boet2020}:
\begin{equation} \label{eq:Msample}
\hN M^2 (\X^{(1)},\ldots,\X^{(N)}) = \frac{1}{N^2} \sum_{j,k=1}^N \prod_{i=1}^n \hN\Psi_i(j,k)
\end{equation}
with
\begin{align}
\notag \hN\Psi_i(j,k) := 
&- \psi_i(X^{(j)}_i- X^{(k)}_i)  - \frac{1}{N^2} \sum_{l,m=1}^N \psi_i(X^{(l)}_i- X^{(m)}_i)\\
 \label{eq:NPsi}&+ \frac{1}{N} \sum_{l=1}^N  \psi_i(X^{(l)}_i- X^{(k)}_i)+ \frac{1}{N} \sum_{m=1}^N \psi_i(X^{(j)}_i- X^{(m)}_i)
\end{align}
and $\psi_i$ given via \eqref{eq:cndf}.
The sample distance multivariance $\hN M$ is uniformly continuous in the sense of \eqref{eq:unif-cont} since the $\psi_i$ are uniformly continuous on $[-1,1]^{d_i}$. Hence by Theorem \ref{thm:consi} the empirical copula distance multivariance $\hN M_{cop}$ inherits the strong consistency of $\hN M.$

We write $\hN \Psi_i\hN T(j,k)$ if the $X_i^{(.)}$ in \eqref{eq:NPsi} are replaced by $\hN T(X_i^{(.)},U_i^{(.)};X_i^{(1)},\ldots,X_i^{(N)})$ and the notation $\hN \Psi_i T_{X_i}(j,k)$ is used for the case were $X_i^{(.)}$ are replaced by $T_{X_i}(X_i^{(.)},U_i^{(.)})$.

For the scaled version $N \cdot \hN M_{cop}^2$ we will show
\begin{equation}\label{eq:Mcop-Plimit}
N \cdot \frac{1}{N^2} \sum_{j,k=1}^N \left(\prod_{i=1}^n \hN\Psi_i \hN T(j,k) - \prod_{i=1}^n \hN\Psi_i T_{X_i}(j,k)\right) \xrightarrow{\Prob} 0, 
\end{equation}
then by Slutsky's Theorem ($Y_n\xrightarrow{d}Y, Z_n\xrightarrow{\Prob} 0 \ \Rightarrow \ Y_n+Z_n \xrightarrow{d} Y$) the statement of the theorem follows. For \eqref{eq:Mcop-Plimit} note that by the Markov inequality it is sufficient to show that the second moment of the left hand side converges to 0.

The second moment of the finite sample version of distance multivariance has been analysed in detail in \cite{BersBoet2018v2}. It is composed of various terms with the coefficients scrupulously collected in \cite[Table 1]{BersBoet2018v2}, this table also indicates the (overall) behaviour of the terms for $N\to \infty$. Based on this we get the following limit (dropping vanishing terms; using the symmetry $...(j,k) = ...(k,j)$ and the identical distribution of the summands in the particular sums):
\begin{equation} \label{eq:limitMcop} \footnotesize
\begin{split}
&\lim_{N\to\infty} \E \left| N \cdot \frac{1}{N^2} \sum_{j,k=1}^N \left(\prod_{i=1}^n \hN\Psi_i \hN T(j,k) - \prod_{i=1}^n \hN\Psi_i T_{X_i}(j,k)\right)\right|^2\\
& = \lim_{N\to\infty} \frac{1}{N^2} \sum_{\substack{j,k,l,m =1\\|\{j,k\}\cap \{l,m\}| = 2}}^n  \E \left[\left(\prod_{i=1}^n \hN\Psi_i \hN T(j,k) - \prod_{i=1}^n \hN\Psi_i T_{X_i}(j,k)\right)^2\right] \\
&\phantom{=} + \lim_{N\to\infty} \frac{1}{N^2} \sum_{\substack{j,k,l,m =1\\j=k, k\neq l, l=m}}^n \E \left[\left(\prod_{i=1}^n \hN\Psi_i \hN T(j,j) - \prod_{i=1}^n \hN\Psi_i T_{X_i}(j,j)\right)\cdot \left(\prod_{i=1}^n \hN\Psi_i \hN T(l,l) - \prod_{i=1}^n \hN\Psi_i T_{X_i}(l,l)\right)\right]\\
& = \lim_{N\to\infty} \frac{1}{N^2} 2 N (N-1) \left[ \prod_{i=1}^n\E\left[(\hN\Psi_i \hN T(1,2))^2\right] +  \prod_{i=1}^n\E\left[(\hN\Psi_i T_{X_i}(1,2))^2\right] -2  \prod_{i=1}^n\E\left[\hN\Psi_i \hN T(1,2) \cdot \hN\Psi_i T_{X_i}(1,2)\right] \right] \\
&\phantom{=} + \lim_{N\to\infty} \frac{1}{N^2} N (N-1) \Bigg[ \prod_{i=1}^n\E\left[\hN\Psi_i \hN T(1,1) \cdot \hN\Psi_i \hN T(2,2)\right]  + \prod_{i=1}^n\E\left[\hN\Psi_i T_{X_i}(1,1) \cdot \hN\Psi_i T_{X_i}(2,2)\right] \\
&\phantom{= + \lim_{N\to\infty} \frac{1}{N^2} N (N-1) \Bigg[}
- \prod_{i=1}^n\E\left[\hN\Psi_i \hN T(1,1) \cdot \hN\Psi_i T_{X_i}(2,2)\right] - \prod_{i=1}^n\E\left[\hN\Psi_i T_{X_i}(1,1) \cdot \hN\Psi_i \hN T(2,2)\right]\Bigg].
\end{split}
\end{equation}
Finally, consider the derived sum of two limits: in the first limit each expectation converges to $M_{\rho_i\otimes \rho_i}(X_1,X_1)^2$ and in the second limit each expectation converges to $[\E(\psi_i(T_{X_i}(X_i,U_i)-T_{X_i}(X_i',U_i')))]^2$. Therefore both limits become 0.

\subsection{Proof of Theorem \ref{thm:dhsiccop} (Asymptotics of $\hN dHSIC_{cop}$)}\label{app:dhsic}
Expanding the product in the definition of $dHSIC$, \eqref{eq:def-Ndhsic}, yields
\begin{align} 
\notag \hN dHSIC(\vx^{(1)},\ldots,\vx^{(n)}) =& \frac{1}{N^2} \sum_{\substack{j=1\\k=1}}^N \prod_{i=1}^n \overline{\psi}_i(x_i^{(j)}-x_i^{(k)})  + \frac{1}{N^{2n}} \sum_{\substack{j_1,\ldots,j_n=1\\k_1,\ldots,k_n=1}}^N \prod_{i=1}^n \overline{\psi}_i(x_i^{(j_i)}-x_i^{(k_i)})\\
\label{eq:Ndhsic-expanded}&- 2 \frac{1}{N^{n+1}} \sum_{\substack{j=1\\k_1,\ldots,k_n=1}}^N \prod_{i=1}^n \overline{\psi}_i(x_i^{(j)}-x_i^{(k_i)}).
\end{align}

The $\psi_i$ are uniformly continuous on $[-1,1]^{d_i}$ and therefore the whole function is uniformly continuous in the sense of \eqref{eq:unif-cont}. Hence by Theorem \ref{thm:consi} the estimator $\hN dHSIC_{cop}$ inherits the consistency of $\hN dHSIC.$

For the scaled version it is sufficient (cf.\ \eqref{eq:Mcop-Plimit}) to show
\begin{equation} \label{eq:dhsiccop-Plimit}
N\cdot \hN dHSIC(T_{\X}(...)) - N\cdot \hN dHSIC_{cop}(...)\xrightarrow{\Prob} 0. 
\end{equation}
Analogously to \eqref{eq:Mcop-Plimit} this convergence could be shown using the Markov inequality. But the method of proof used for distance multivariance does not transfer directly, since in the case of $\hN dHSIC$ individual sums (similar to those appearing in \eqref{eq:limitMcop}) diverge in the limit. Only in a joint analysis these diverging terms cancel explicitly. The remaining terms converge and cancel in the limit as in \eqref{eq:limitMcop}. We skip the details here, but provide a sketch of a closely related alternative approach (which is also tedious): In \cite{PfisBuehSchoPete2017} the variance of $\hN dHSIC$ was calculated. An analogous formula can also be derived for the covariance of the two estimators in \eqref{eq:dhsiccop-Plimit}. Hereto the $e_.$ in \cite[Proposition 5]{PfisBuehSchoPete2017} have to be replaced by
\begin{align}
e_0^2(j)&:=\E(\overline\psi_j(Y_j-Y_j') \overline\psi_j(Z_j''-Z_j''')),\\
e_1(j)&:=\E(\overline\psi_j(Y_j-Y_j') \overline\psi_j(Z_j-Z_j')),\\
e_2(j)&:=\E(\overline\psi_j(Y_j-Y_j') \overline\psi_j(Z_j'-Z_j''))
\end{align}
where $Y_j^\cdot:=T_{X_j}(X_j^\cdot,U_j^\cdot)$ and $Z_j^\cdot:=\hN T(X_j^\cdot,U_j^\cdot; X^{(1)},\ldots,X^{(N)})$ and variables with different upper indicies denote independent copies. The variance of $\hN dHSIC_{cop}$ is obtained if $Y_j^\cdot:=\hN T(X_j^\cdot,U_j^\cdot; X_j^{(1)},\ldots,X_j^{(N)})$ is used instead. All these expectations are bounded, the $N$ dependent versions converge by dominated convergence to those with $T_{X_i}$ instead of $\hN T.$ Hence the overall limit becomes
\begin{equation}
Var\bigg(N\cdot \hN dHSIC(T_{\X}(...)) - N\cdot \hN dHSIC_{cop}(...)\bigg)\xrightarrow{a.s.} 0.
\end{equation}
Note that also the expectation converges to 0, thus \eqref{eq:dhsiccop-Plimit} follows by Chebychev's inequality.

\newpage
\renewcommand{\thesection}{\Alph{section}}
\setcounter{section}{18}
\counterwithin{figure}{section}
\counterwithin{table}{section}
\section{Supplement}\label{sec:supp}
This supplement complements the discussion of Section \ref{sec:ex}. 
It includes:
\begin{itemize}
\item Figure \ref{fig:copulas}: The visualized dependence corresponding to the copulas.
\item Figure \ref{fig:MCU-vs-MC}: A comparison of the Monte Carlo p-value derivation method of Remark \ref{rem:speed}.
\item Table \ref{tab:fullGene}: The full table corresponding to Table \ref{tab:comp}.
\item Tables \ref{tab:vary-N}, \ref{tab:vary-n} and \ref{tab:vary-tau}: Simulations in the setting of Table \ref{tab:comp} for $(N,n,\tau)$ in $\{ (50,5,0.1),(200,5,0.1),$\\$ (100,2,0.1),(100,10,0.1),(100,5,0.05),(100,5,0.2)\}.$ 
\item Table \ref{tab:dhsic-variants}: A comparison of $dHSIC_{cop}$ with $\delta \in \{0.1,0.2, 0.5,0.75,1,2,3,4,5\}.$
\end{itemize}

New terms in the following tables are: the independence copula (indep) providing $H_0$ samples, uniform marginals (U), Bernoulli marginals (B) and $dHSIC^\delta_{cop}$ for some $\delta >0$ were $\delta$ denotes explicitly the used parameter. In the tables also tests based on the classical, i.e., without the distributional transform, measures $\overline{\Mskript}$ and $dHSIC$ are included. For the p-values of $\overline{\Mskript}$ Pearson's approximation and for $dHSIC$ the eigenvalue method was used. Note that $\overline{\Mskript}$ requires some finite moments for its proven asymptotics, here we perform the tests regardless of the required assumptions.

Main observations:
\begin{itemize}
\item The approximate Monte Carlo $H_0$ samples of Remark \ref{rem:speed} provide good p-value estimates. For multivariance these become conservative for large dimensions and for small sample sizes, in the same settings these become liberal for dHSIC (Figure \ref{fig:MCU-vs-MC}).
\item In the case of Bernoulli marginals all tables indicate that the classical methods are superior (Tables \ref{tab:fullGene}, \ref{tab:vary-N}, \ref{tab:vary-n}).
\item The optimal measure depends strongly on the given marginals and copula (Tables \ref{tab:fullGene}, \ref{tab:vary-N}, \ref{tab:vary-n}, \ref{tab:dhsic-variants}). Moreover, the heuristic measure selection of dHSIC does not perform well for uniform marginals (Table \ref{tab:dhsic-variants}).
\end{itemize}

\begin{figure}[H]
\centering 
\setlength{\myl}{0.31\textwidth}
\includegraphics[width=\myl]{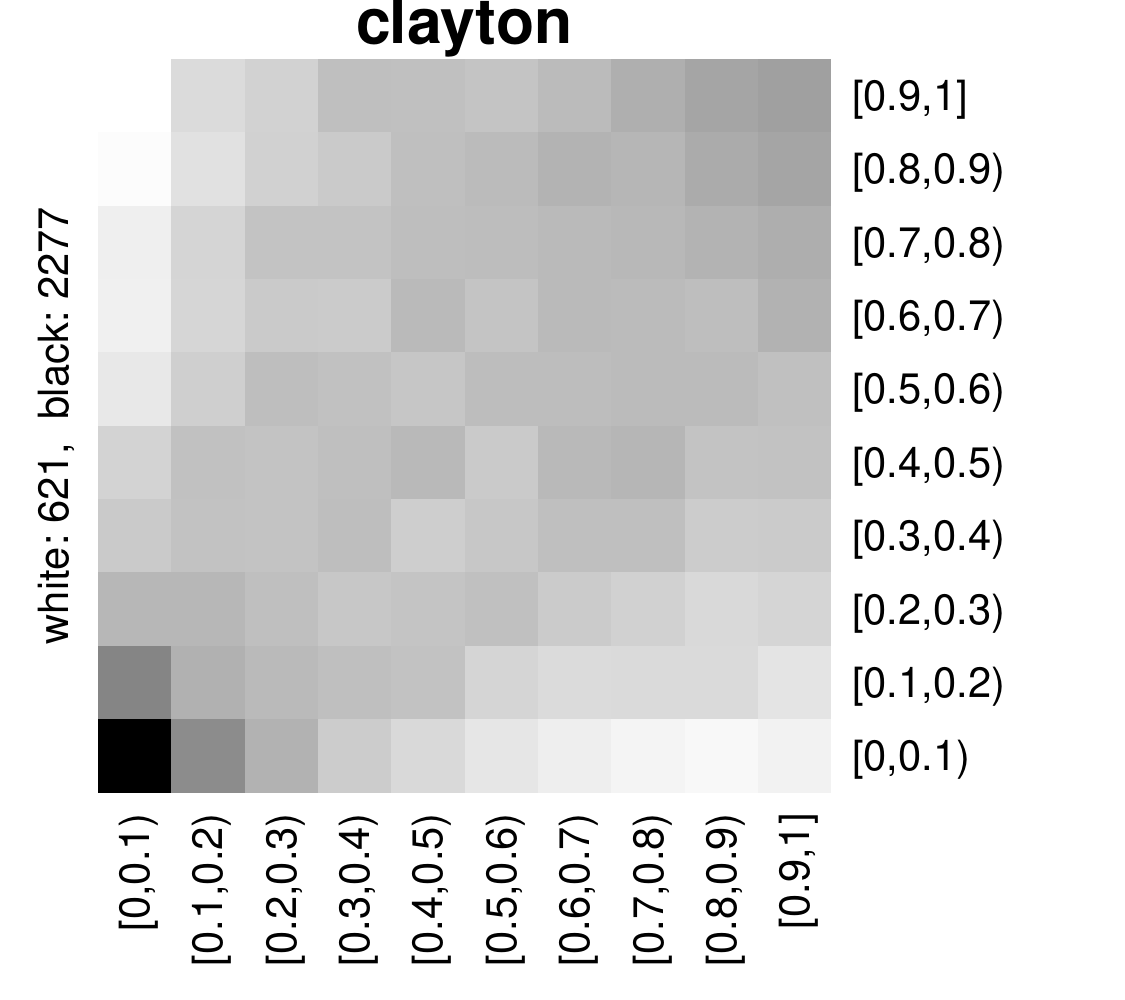} 
\includegraphics[width=\myl]{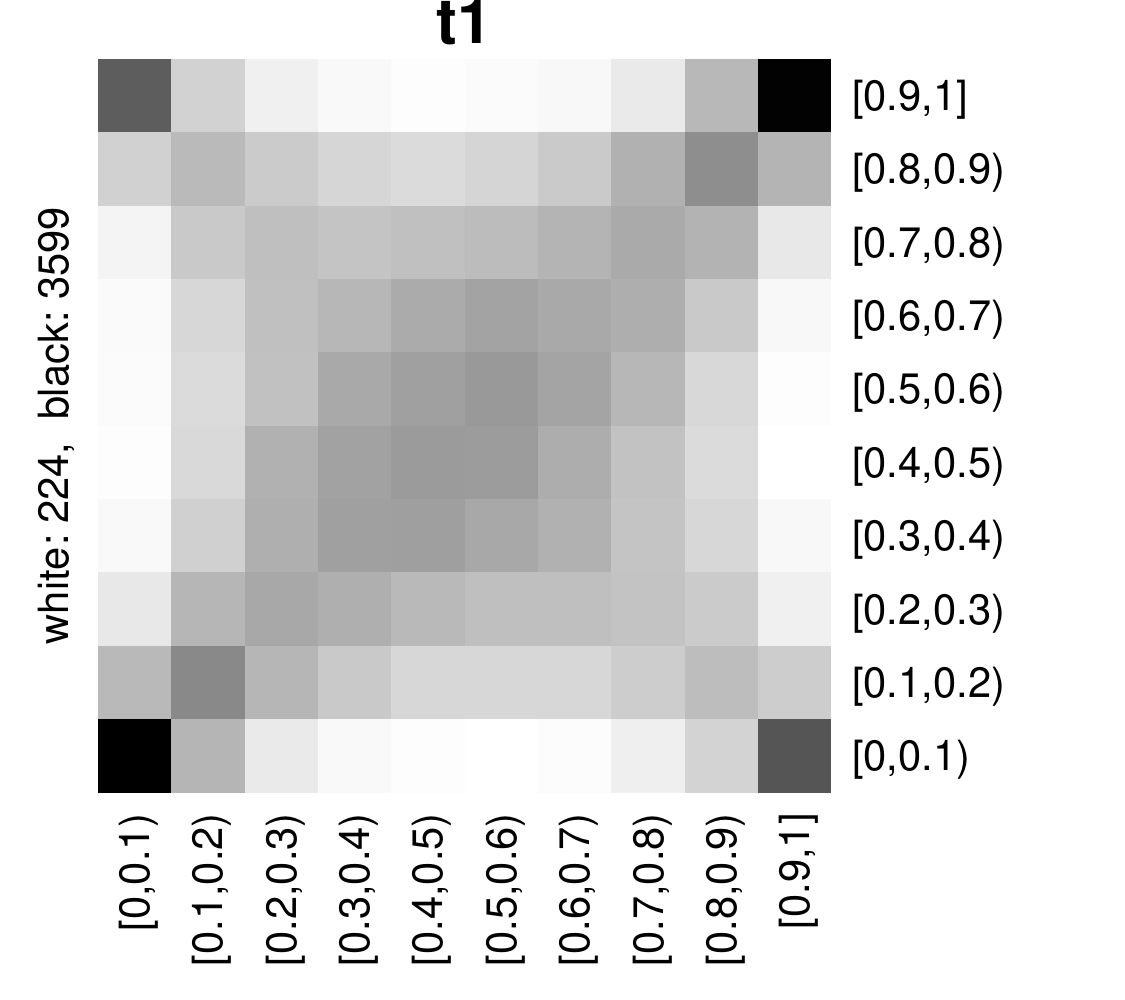} 
\includegraphics[width=\myl]{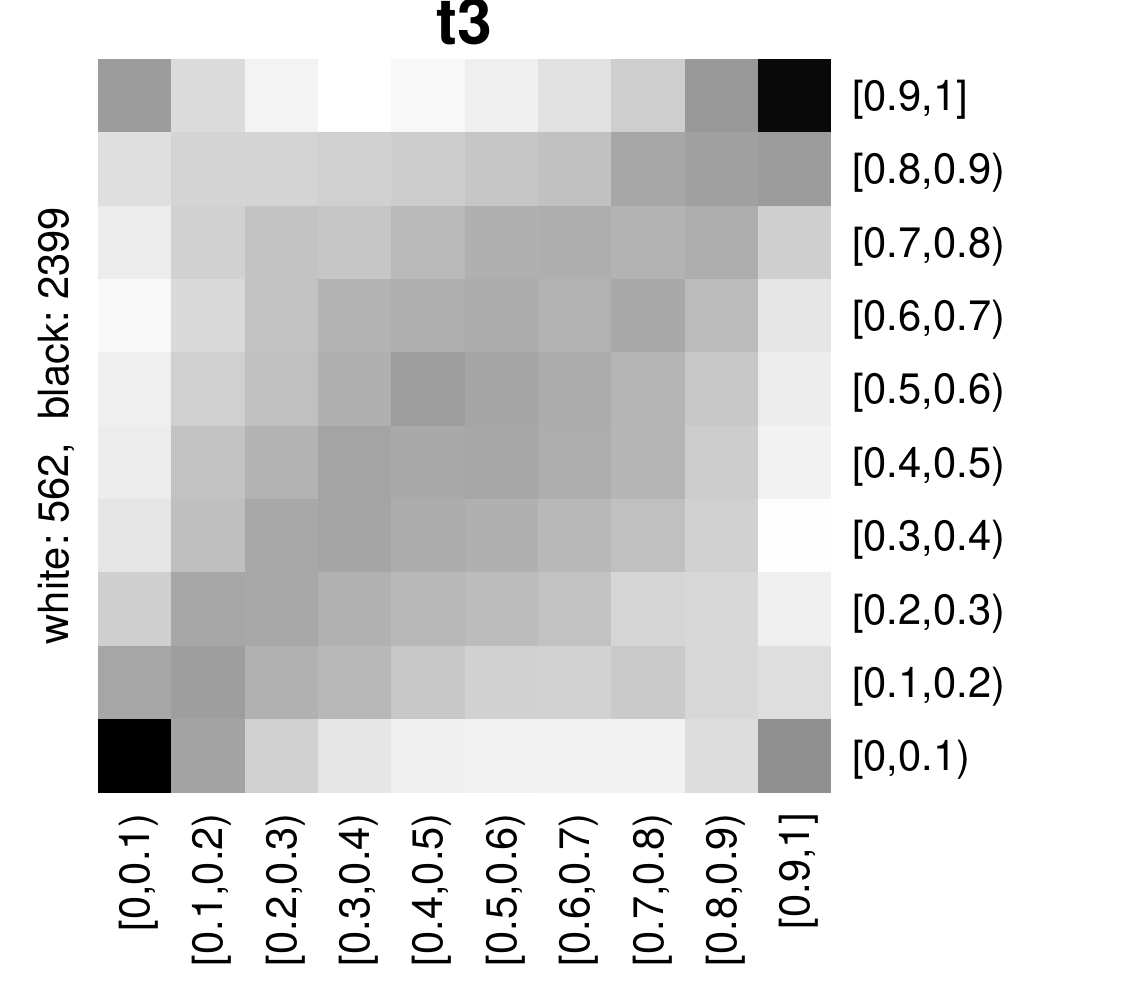} 
\includegraphics[width=\myl]{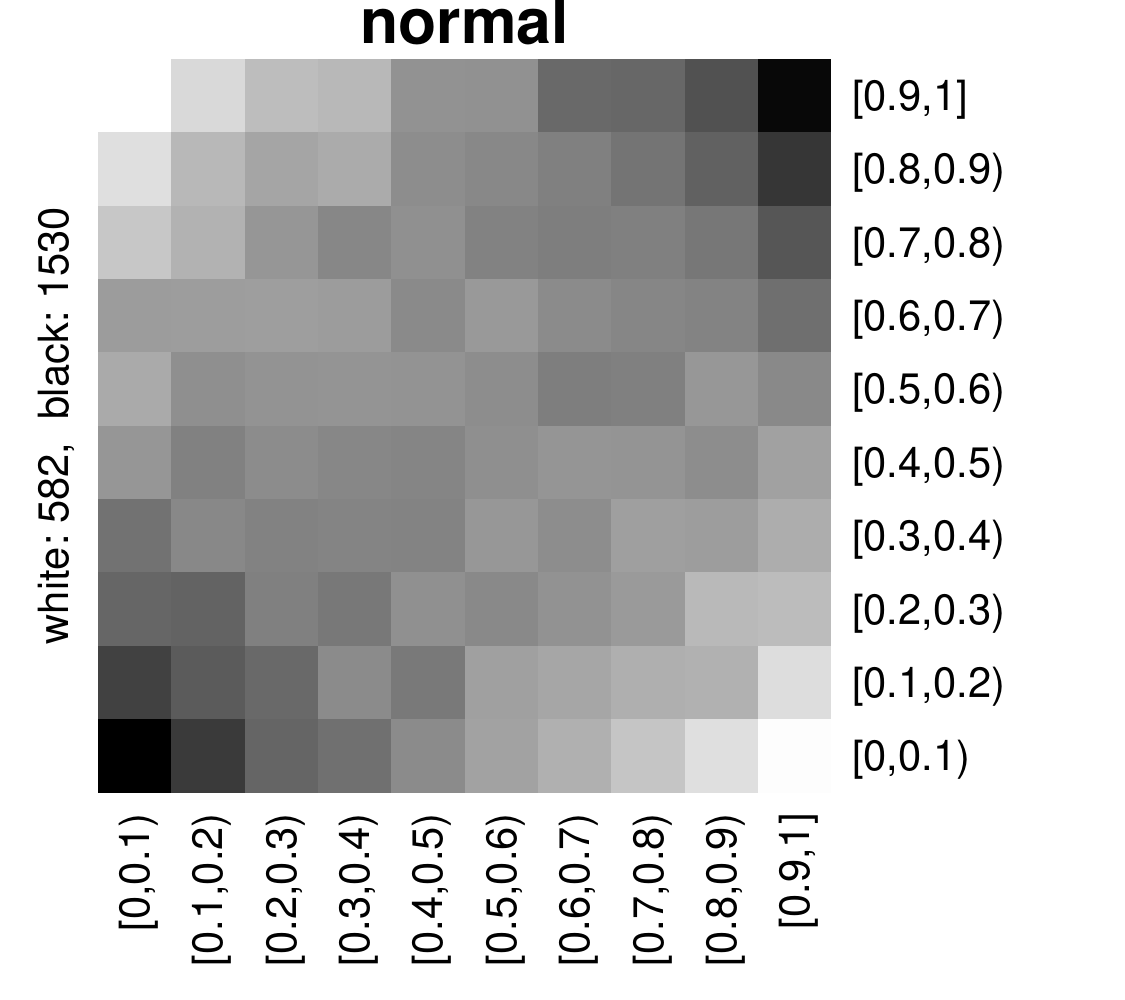} 
\includegraphics[width=\myl]{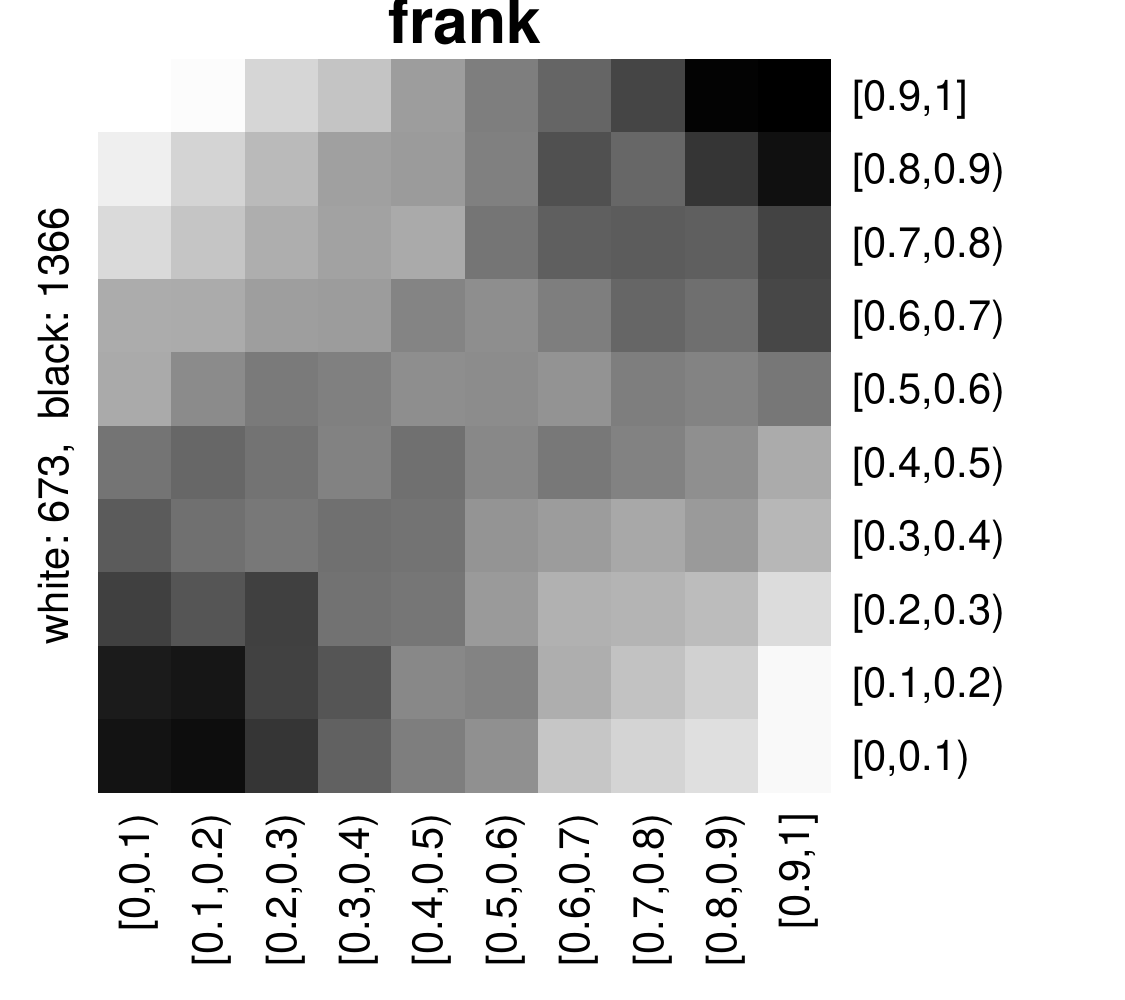} 
\includegraphics[width=\myl]{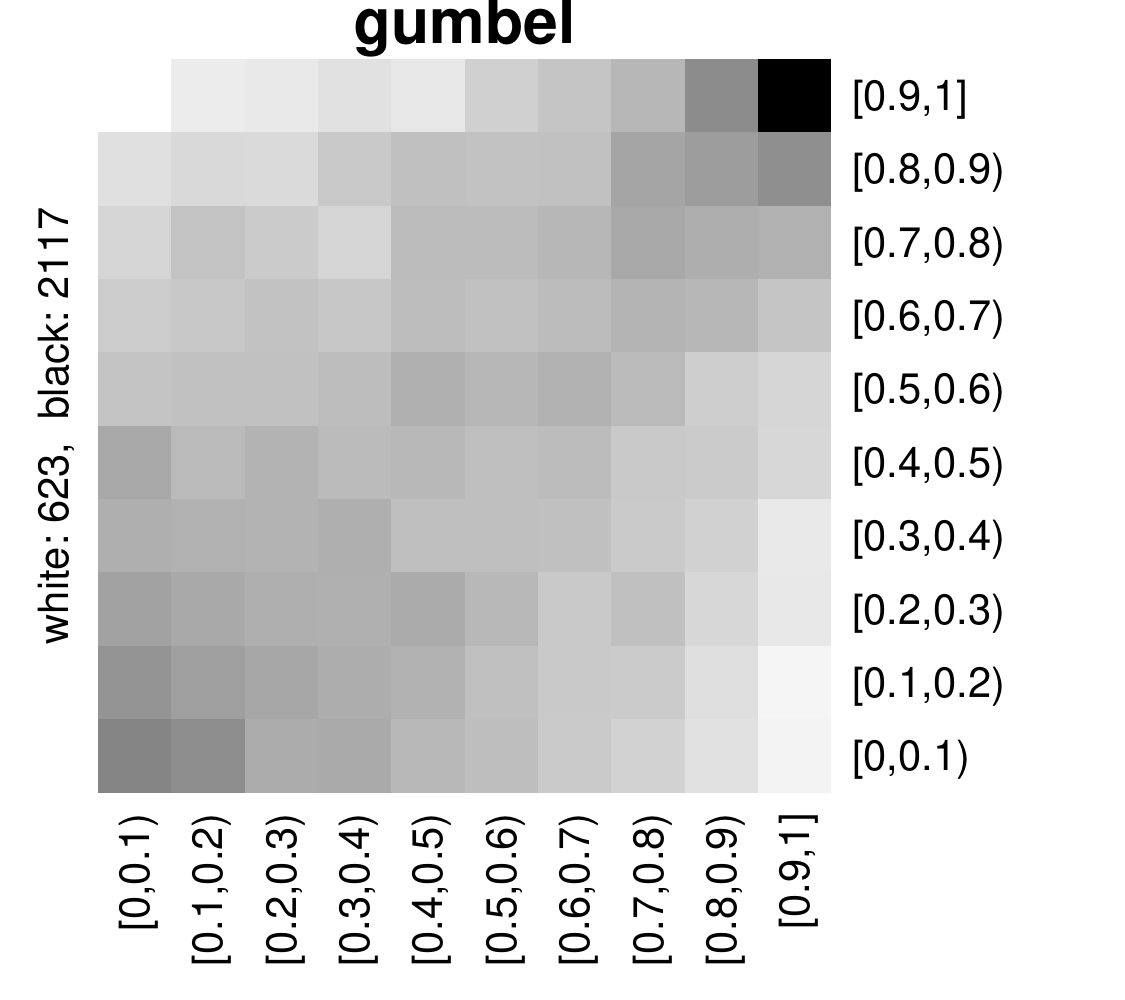} 

\caption{Induced copula dependence. Depicted are the bivariate marginals in the case $N=100, n = 5$ with pairwise Kendall's tau of 0.1, based on 100000 samples and binned into the stated ranges. Each plot has its own colour range from white to black in 255 steps, the bin with the smallest frequency is white and the one with the largest frequency is black, the corresponding frequencies are stated on the left.}
\label{fig:copulas}
\end{figure}

\begin{figure}[H]
\centering
\includegraphics[width=0.8\textwidth]{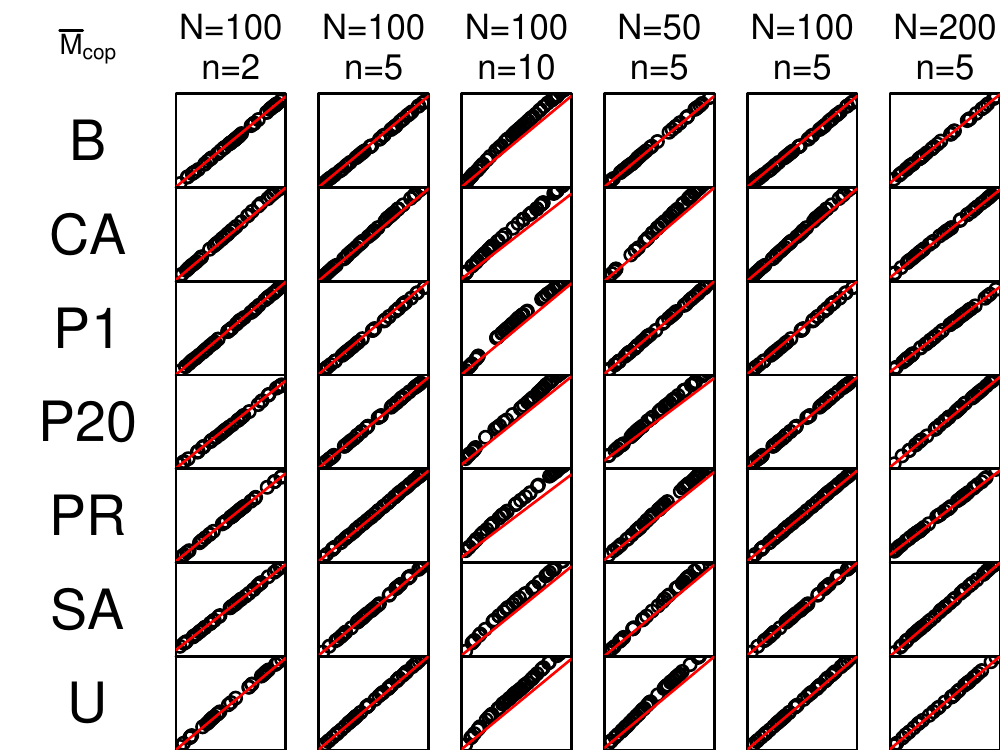} 
\vspace{1\baselineskip}

\includegraphics[width=0.8\textwidth]{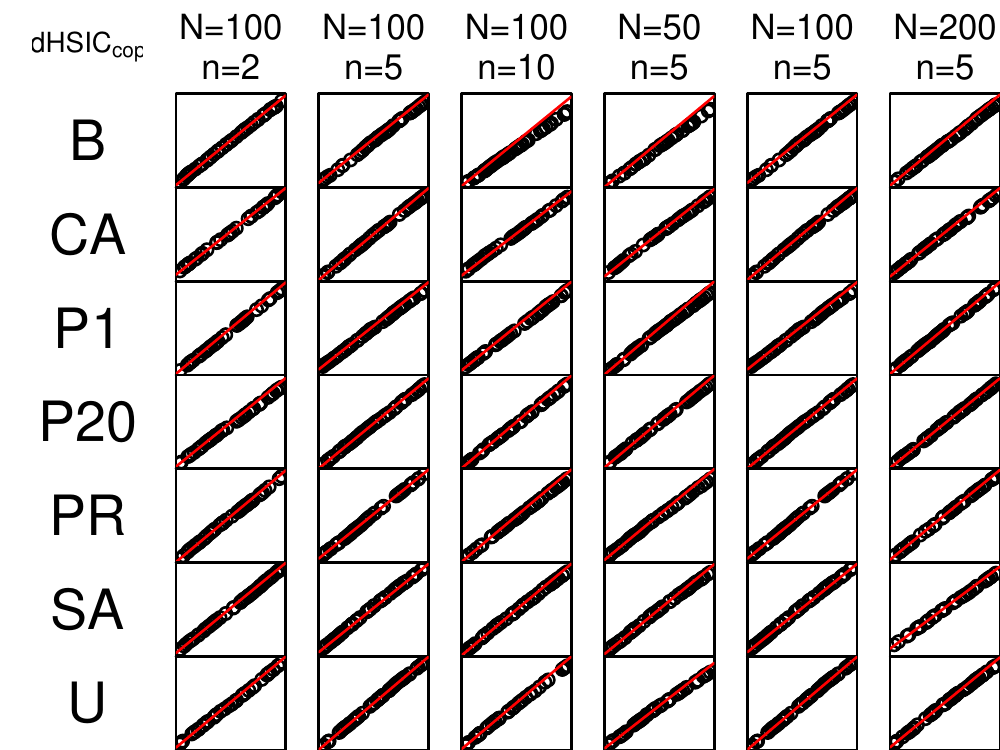} 
\caption{Performance under $H_0$ of the fast p-value generation discussed in Remark \ref{rem:speed}. The plots are similar to those in Figure \ref{fig:H0-speed}. For each graph 1000 samples of $H_0$ data (of the stated marginals, sample size, dimension) were considered and its \textit{exact Monte Carlo} p-values ($x$-axis) were computed using 100000 samples of $H_0$ data, and its fast \textit{approximate Monte Carlo} p-values ($y$-axis) were computed also based on 100000 samples (as described in Remark \ref{rem:speed}). The $x$- and $y$-axes range both from 0 to 0.05. Note that points above the line indicate conservative behaviour. \newline
Observations: Overall the performance of the approximate method is very good. Nevertheless, the method behaves complementary for dHSIC and multivariance: For multivariance the fast Monte Carlo method is conservative for small samples and also for large dimensions. In the same settings it becomes (slightly) liberal for dHSIC.}
\label{fig:MCU-vs-MC}
\end{figure}

\begin{table}[ht]
\centering
\footnotesize
\begin{tabular}{lllllllllllll}
 copula & type & $\overline{\Mskript}$ & $\overline{\Mskript}_{cop}$ & $dHSIC$ & $dHSIC_{cop}$ & Sn & Tn & TnT & JSn & JTn & JTnT & Wn \\ 
  \hline
normal & B & 36.1 & 29.6 & \textbf{50.2 } & 30.7 &  &  &  &  &  &  &  \\ 
   & CA & 4.1 & 72.9 & 3.4 & \textbf{83.2 } & 82.4 & 67.8 & 82.5 & 81.7 & 67.9 & 82.6 & 72.3 \\ 
   & P1 & 71.0 & 53.1 & 63.3 & 63.1 & \textbf{82.5 } & 61.6 & 75.1 & 68.6 & 49.4 & 65.9 & 60.3 \\ 
   & P20 & 74.1 & 72.6 & 1.0 & \textbf{82.7 } & 82.4 & 67.4 & 81.7 & 80.5 & 66.2 & 81.6 & 71.6 \\ 
   & RP & 0.8 & 71.5 & 1.8 & 81.1 & \textbf{83.6 } & 68.5 & 82.5 & 81.7 & 66.8 & 82.1 & 71.7 \\ 
   & SA & 39.6 & 72.9 & 44.7 & \textbf{83.0 } & 82.4 & 68.1 & 82.2 & 81.7 & 68.2 & 82.5 & 72.3 \\ 
   & U & 73.6 & 73.0 & 82.1 & \textbf{83.3 } &  &  &  &  &  &  &  \\ 
  t1 & B & 35.3 & 28.3 & \textbf{49.7 } & 31.2 &  &  &  &  &  &  &  \\ 
   & CA & \textbf{100 } & \textbf{100 } & \textbf{100 } & 83.1 & 78.6 & 82.9 & 97.7 & 79.4 & 83.5 & 97.5 & 99.2 \\ 
   & P1 & \textbf{99.9 } & 90.2 & 74.6 & 65.3 & 82.1 & 62.0 & 80.1 & 68.7 & 50.3 & 69.1 & 93.4 \\ 
   & P20 & \textbf{100 } & \textbf{100 } & \textbf{100 } & 82.7 & 78.6 & 83.0 & 98.0 & 78.5 & 81.9 & 97.4 & 99.2 \\ 
   & RP & 48.8 & \textbf{100 } & 2.6 & 81.0 & 82.9 & 81.0 & 96.9 & 80.9 & 80.2 & 96.3 & 98.6 \\ 
   & SA & \textbf{100 } & \textbf{100 } & 99.8 & 82.8 & 78.7 & 83.0 & 97.9 & 79.4 & 83.8 & 97.5 & 99.2 \\ 
   & U & \textbf{100 } & \textbf{100 } & 54.0 & 82.5 &  &  &  &  &  &  &  \\ 
  t3 & B & 34.3 & 27.7 & \textbf{47.6 } & 28.6 &  &  &  &  &  &  &  \\ 
   & CA & 92.9 & 96.5 & \textbf{99.9 } & 82.6 & 78.3 & 66.7 & 82.2 & 77.8 & 66.8 & 82.9 & 93.4 \\ 
   & P1 & \textbf{97.6 } & 71.0 & 64.5 & 67.0 & 81.7 & 58.8 & 72.9 & 68.3 & 47.9 & 66.1 & 81.8 \\ 
   & P20 & \textbf{99.9 } & 96.3 & 69.6 & 82.3 & 77.6 & 66.2 & 81.7 & 76.8 & 65.7 & 81.4 & 92.9 \\ 
   & RP & 16.2 & \textbf{94.1 } & 1.0 & 81.9 & 80.0 & 66.6 & 81.3 & 78.0 & 64.5 & 81.4 & 91.9 \\ 
   & SA & \textbf{100 } & 96.6 & 94.5 & 82.8 & 77.9 & 67.0 & 82.3 & 77.9 & 66.6 & 82.4 & 93.6 \\ 
   & U & \textbf{96.5 } & \textbf{96.5 } & 69.5 & 82.6 &  &  &  &  &  &  &  \\ 
  clayton & B & 31.6 & 27.8 & \textbf{46.5 } & 29.4 &  &  &  &  &  &  &  \\ 
   & CA & 17.7 & 79.7 & 11.7 & \textbf{85.5 } & 70.4 & 67.9 & 83.3 & 70.3 & 68.2 & 83.3 & 75.2 \\ 
   & P1 & 46.2 & 48.2 & 43.2 & 53.4 & 68.8 & 54.4 & \textbf{69.5 } & 56.5 & 42.3 & 58.5 & 47.1 \\ 
   & P20 & 79.3 & 77.7 & 1.3 & \textbf{84.7 } & 71.5 & 67.9 & 83.6 & 71.2 & 66.8 & 83.9 & 75.3 \\ 
   & RP & 0.0 & 75.2 & 1.8 & \textbf{81.8 } & 72.8 & 67.4 & 81.4 & 69.9 & 65.0 & 81.4 & 71.2 \\ 
   & SA & 67.2 & 79.7 & 59.5 & \textbf{85.3 } & 70.2 & 67.5 & 83.4 & 69.9 & 68.4 & 83.5 & 75.5 \\ 
   & U & 79.3 & 79.5 & 81.9 & \textbf{85.4 } &  &  &  &  &  &  &  \\ 
  frank & B & 53.8 & 39.4 & \textbf{62.6 } & 37.0 &  &  &  &  &  &  &  \\ 
   & CA & 1.5 & 81.8 & 3.8 & 85.7 & 68.8 & 78.9 & 85.5 & 68.6 & 79.0 & \textbf{85.8 } & 80.7 \\ 
   & P1 & 71.8 & 65.3 & 61.7 & 65.6 & 67.3 & 70.7 & \textbf{79.9 } & 54.1 & 56.9 & 70.3 & 70.3 \\ 
   & P20 & 76.2 & 82.3 & 2.4 & 84.9 & 69.3 & 79.0 & \textbf{86.1 } & 68.2 & 77.8 & 85.2 & 80.1 \\ 
   & RP & 0.3 & 80.8 & 0.2 & 84.9 & 70.3 & 78.9 & \textbf{85.7 } & 67.8 & 77.7 & 85.5 & 79.8 \\ 
   & SA & 27.9 & 81.8 & 54.0 & \textbf{85.9 } & 68.8 & 78.8 & 85.7 & 68.4 & 79.2 & 85.4 & 80.3 \\ 
   & U & 82.8 & 81.8 & 83.4 & \textbf{85.7 } &  &  &  &  &  &  &  \\ 
  gumbel & B & 41.2 & 31.2 & \textbf{44.7 } & 27.9 &  &  &  &  &  &  &  \\ 
   & CA & 27.2 & 85.9 & 7.9 & 81.5 & 60.6 & 81.9 & 80.5 & 60.0 & 82.4 & 80.8 & \textbf{88.4 } \\ 
   & P1 & \textbf{88.4 } & 73.9 & 54.4 & 65.2 & 61.9 & 79.2 & 76.0 & 48.5 & 67.6 & 67.1 & 84.4 \\ 
   & P20 & \textbf{89.0 } & 85.4 & 1.3 & 80.8 & 62.4 & 81.6 & 80.3 & 60.4 & 80.9 & 79.1 & 88.0 \\ 
   & RP & 21.1 & 85.8 & 0.2 & 80.6 & 62.6 & 83.0 & 80.2 & 60.4 & 81.5 & 79.7 & \textbf{87.9 } \\ 
   & SA & 72.0 & 85.8 & 38.1 & 81.6 & 60.6 & 82.0 & 80.1 & 59.4 & 82.1 & 80.6 & \textbf{88.3 } \\ 
   & U & \textbf{86.0 } & 85.9 & 77.5 & 81.4 &  &  &  &  &  &  &  \\ 
  indep & B & 4.2 & \textbf{4.8 } & 4.3 & \textbf{4.8 } &  &  &  &  &  &  &  \\ 
   & CA & 1.0 & 5.0 & 1.5 & \textbf{5.3 } &  &  &  &  &  &  &  \\ 
   & P1 & 4.3 & 4.7 & 2.8 & \textbf{5.4 } &  &  &  &  &  &  &  \\ 
   & P20 & \textbf{5.6 } & 5.3 & 0.2 & 5.1 &  &  &  &  &  &  &  \\ 
   & RP & 0.0 & \textbf{5.6 } & 0.0 & 4.7 &  &  &  &  &  &  &  \\ 
   & SA & 3.7 & 4.8 & 1.7 & \textbf{4.9 } &  &  &  &  &  &  &  \\ 
   & U & 5.0 & 4.9 & 4.1 & \textbf{5.1 } &  &  &  &  &  &  &  \\ 
   \hline
\end{tabular}

\caption{Extension of Table \ref{tab:comp} with details of \cite[Table S15]{GeneNesRemiMurp2019}. The maximum in each row is printed in bold. The quoted table contained one more test using a measure called 'R' based on \cite{GeneNesRemi2013}. It was removed here because in general it is non-consistent. Nevertheless, in the given setting its power was only exceeded by multivariance in the case of the Student copulas.\newline
Observations: There is no overall optimal test. The classical measures outperform sometimes the copula based measures, in particular in the case of the Student copula (with 1 and 3 degrees of freedom).} 
\label{tab:fullGene}
\end{table}

\begin{table}[ht]
\centering
\begin{tabular}{llllll}
  & &\multicolumn{4}{c}{$N = 50, n = 5$} \\
copula & type & $\overline{\Mskript}$ & $\overline{\Mskript}_{cop}$ & $dHSIC$ & $dHSIC_{cop}$ \\ 
  \hline
normal & B & 20.1 & 14.0 & \textbf{25.8 } & 15.9 \\ 
   & CA & 5.4 & 42.4 & 0.3 & \textbf{53.8 } \\ 
   & P1 & \textbf{41.6 } & 27.9 & 26.9 & 34.4 \\ 
   & P20 & 39.7 & 41.2 & 0.0 & \textbf{53.2 } \\ 
   & RP & 2.1 & 41.9 & 0.2 & \textbf{52.2 } \\ 
   & SA & 19.7 & 41.4 & 11.3 & \textbf{53.2 } \\ 
   & U & 44.7 & 41.6 & 47.3 & \textbf{53.8 } \\ 
  t1 & B & 19.9 & 15.9 & \textbf{25.4 } & 16.1 \\ 
   & CA & \textbf{100 } & 99.8 & 99.8 & 59.2 \\ 
   & P1 & \textbf{97.8 } & 69.1 & 24.3 & 39.3 \\ 
   & P20 & \textbf{100 } & 99.7 & 94.4 & 57.9 \\ 
   & RP & 44.6 & \textbf{99.2 } & 0.1 & 57.9 \\ 
   & SA & \textbf{100 } & 99.8 & 68.4 & 59.0 \\ 
   & U & 99.6 & \textbf{99.8 } & 22.8 & 58.7 \\ 
  t3 & B & 18.5 & 16.1 & \textbf{27.2 } & 17.8 \\ 
   & CA & \textbf{84.6 } & 79.4 & 82.2 & 57.1 \\ 
   & P1 & \textbf{80.6 } & 44.5 & 23.6 & 38.6 \\ 
   & P20 & \textbf{97.8 } & 79.6 & 6.9 & 57.3 \\ 
   & RP & 18.5 & \textbf{71.9 } & 0.3 & 55.0 \\ 
   & SA & \textbf{99.0 } & 80.2 & 32.9 & 57.1 \\ 
   & U & \textbf{81.5 } & 79.9 & 36.8 & 57.1 \\ 
  clayton & B & 16.6 & 13.5 & \textbf{23.1 } & 16.5 \\ 
   & CA & 17.9 & 46.5 & 2.7 & \textbf{54.6 } \\ 
   & P1 & 23.8 & 21.8 & 15.8 & \textbf{27.7 } \\ 
   & P20 & 50.1 & 45.7 & 0.1 & \textbf{53.4 } \\ 
   & RP & 0.6 & 42.4 & 0.3 & \textbf{51.5 } \\ 
   & SA & 42.3 & 46.3 & 16.2 & \textbf{54.7 } \\ 
   & U & 49.5 & 46.4 & 44.9 & \textbf{54.1 } \\ 
  frank & B & 26.9 & 18.8 & \textbf{29.3 } & 19.4 \\ 
   & CA & 2.3 & 44.6 & 1.3 & \textbf{51.3 } \\ 
   & P1 & \textbf{38.9 } & 31.9 & 23.0 & 35.1 \\ 
   & P20 & 35.4 & 43.5 & 0.2 & \textbf{50.8 } \\ 
   & RP & 0.5 & 43.2 & 0.0 & \textbf{51.1 } \\ 
   & SA & 14.0 & 44.4 & 15.4 & \textbf{50.7 } \\ 
   & U & 47.0 & 44.8 & 47.4 & \textbf{51.4 } \\ 
  gumbel & B & 23.0 & 14.9 & \textbf{23.4 } & 16.1 \\ 
   & CA & 23.9 & \textbf{56.4 } & 1.6 & 51.4 \\ 
   & P1 & \textbf{66.0 } & 44.7 & 19.8 & 38.0 \\ 
   & P20 & \textbf{60.0 } & 56.0 & 0.0 & 51.2 \\ 
   & RP & 19.6 & \textbf{56.1 } & 0.0 & 51.1 \\ 
   & SA & 45.8 & \textbf{56.3 } & 9.6 & 51.0 \\ 
   & U & \textbf{59.6 } & 56.3 & 41.8 & 51.3 \\ 
  indep & B & \textbf{4.8 } & 3.9 & 3.2 & 4.6 \\ 
   & CA & 2.2 & 3.2 & 0.2 & \textbf{4.7 } \\ 
   & P1 & 6.1 & 5.5 & 2.4 & \textbf{6.8 } \\ 
   & P20 & \textbf{5.6 } & 4.2 & 0.1 & 5.4 \\ 
   & RP & 0.2 & 3.4 & 0.0 & \textbf{5.3 } \\ 
   & SA & \textbf{4.9 } & 3.5 & 0.7 & 4.6 \\ 
   & U & 4.5 & 3.1 & 2.7 & \textbf{4.7 } \\ 
   \hline
\end{tabular}

\begin{tabular}{llll}
  \multicolumn{4}{c}{$N = 200, n = 5$} \\
$\overline{\Mskript}$ & $\overline{\Mskript}_{cop}$ & $dHSIC$ & $dHSIC_{cop}$ \\ 
  \hline
69.5 & 59.6 & \textbf{84.4 } & 60.8 \\ 
  2.9 & 97.6 & 12.3 & \textbf{99.3 } \\ 
  95.4 & 88.4 & \textbf{96.2 } & 93.4 \\ 
  98.1 & 97.1 & 15.9 & \textbf{99.0 } \\ 
  0.3 & 97.0 & 14.9 & \textbf{99.0 } \\ 
  80.7 & 97.6 & 92.8 & \textbf{99.3 } \\ 
  97.9 & 97.5 & \textbf{99.3 } & \textbf{99.3 } \\ 
  66.4 & 54.7 & \textbf{82.1 } & 56.1 \\ 
  \textbf{100 } & \textbf{100 } & \textbf{100 } & 96.7 \\ 
  \textbf{100 } & 98.9 & 99.6 & 90.9 \\ 
  \textbf{100 } & \textbf{100 } & \textbf{100 } & 96.4 \\ 
  48.6 & \textbf{100 } & 22.3 & 96.5 \\ 
  \textbf{100 } & \textbf{100 } & \textbf{100 } & 96.8 \\ 
  \textbf{100 } & \textbf{100 } & 87.8 & 96.7 \\ 
  67.9 & 57.3 & \textbf{84.3 } & 58.2 \\ 
  96.7 & 99.9 & \textbf{100 } & 99.0 \\ 
  \textbf{99.9 } & 94.0 & 97.5 & 92.3 \\ 
  \textbf{100 } & \textbf{100 } & 99.9 & 99.1 \\ 
  20.4 & \textbf{99.7 } & 15.2 & 98.8 \\ 
  \textbf{100 } & 99.8 & \textbf{100 } & 99.0 \\ 
  \textbf{100 } & 99.9 & 97.4 & 99.0 \\ 
  64.4 & 53.3 & \textbf{80.1 } & 56.2 \\ 
  19.5 & 97.9 & 42.5 & \textbf{99.4 } \\ 
  83.1 & 83.2 & 82.6 & \textbf{87.7 } \\ 
  98.6 & 97.3 & 22.1 & \textbf{99.4 } \\ 
  0.0 & 96.5 & 18.2 & \textbf{99.1 } \\ 
  94.7 & 97.8 & 97.3 & \textbf{99.4 } \\ 
  97.5 & 97.9 & 99.3 & \textbf{99.4 } \\ 
  85.8 & 72.8 & \textbf{90.9 } & 69.2 \\ 
  0.5 & 98.8 & 16.4 & \textbf{99.3 } \\ 
  \textbf{94.8 } & 92.5 & 93.0 & 93.4 \\ 
  97.7 & 98.5 & 31.5 & \textbf{99.1 } \\ 
  0.1 & 98.4 & 2.7 & \textbf{99.2 } \\ 
  71.0 & 98.7 & 94.4 & \textbf{99.3 } \\ 
  98.9 & 98.9 & \textbf{99.3 } & \textbf{99.3 } \\ 
  79.6 & 60.5 & \textbf{80.5 } & 56.0 \\ 
  32.7 & \textbf{98.9 } & 26.6 & 98.3 \\ 
  \textbf{99.3 } & 95.3 & 93.3 & 93.3 \\ 
  \textbf{99.6 } & 99.0 & 16.3 & 98.4 \\ 
  21.6 & \textbf{98.9 } & 1.7 & 98.3 \\ 
  95.0 & \textbf{98.8 } & 85.8 & 98.5 \\ 
  \textbf{98.9 } & \textbf{98.9 } & 98.0 & 98.3 \\ 
  4.3 & 5.0 & 5.0 & \textbf{5.2 } \\ 
  0.3 & 5.7 & 2.0 & \textbf{6.0 } \\ 
  5.0 & \textbf{5.6 } & 4.9 & 5.2 \\ 
  \textbf{6.2 } & 5.9 & 0.4 & 6.0 \\ 
  0.0 & 5.9 & 0.0 & \textbf{6.5 } \\ 
  5.0 & \textbf{5.7 } & 3.2 & \textbf{5.7 } \\ 
  5.8 & 5.7 & 5.7 & \textbf{6.0 } \\ 
   \hline
\end{tabular}

\medskip
\caption{Setting of Table \ref{tab:comp} with different $N$. For each parameter setting the  maximum in each row is printed in bold. \newline
 Observations: With increasing sample size the power increases and the optimal test changes in certain cases. The tests without the distributional transform can be more powerful than the copula versions. The copula-based tests seem to become somewhat liberal with increasing sample size.} 
\label{tab:vary-N}
\end{table}

\begin{table}[ht]
\centering
\begin{tabular}{llllll}
  & &\multicolumn{4}{c}{$N = 100, n = 2$} \\
copula & type & $\overline{\Mskript}$ & $\overline{\Mskript}_{cop}$ & $dHSIC$ & $dHSIC_{cop}$ \\ 
  \hline
normal & B & 17.3 & 14.1 & \textbf{17.4 } & 13.7 \\ 
   & CA & 7.2 & 28.7 & 5.2 & \textbf{32.9 } \\ 
   & P1 & \textbf{23.7 } & 20.2 & 23.0 & 21.0 \\ 
   & P20 & 29.1 & 29.0 & 12.3 & \textbf{33.3 } \\ 
   & RP & 5.6 & 28.6 & 8.7 & \textbf{33.2 } \\ 
   & SA & 28.8 & 28.6 & 23.2 & \textbf{33.1 } \\ 
   & U & 27.9 & 28.7 & \textbf{32.6 } & \textbf{32.6 } \\ 
  t1 & B & 15.5 & 12.3 & \textbf{15.7 } & 11.1 \\ 
   & CA & \textbf{99.0 } & 40.5 & 96.3 & 28.6 \\ 
   & P1 & \textbf{46.3 } & 24.0 & 19.8 & 22.2 \\ 
   & P20 & 75.7 & 40.6 & \textbf{95.6 } & 27.1 \\ 
   & RP & \textbf{40.3 } & 37.4 & 17.8 & 28.3 \\ 
   & SA & \textbf{97.2 } & 40.6 & 38.6 & 28.3 \\ 
   & U & \textbf{40.7 } & 40.6 & 17.1 & 28.4 \\ 
  t3 & B & 15.3 & 13.0 & \textbf{15.4 } & 11.0 \\ 
   & CA & \textbf{58.3 } & 29.8 & 23.6 & 30.7 \\ 
   & P1 & \textbf{31.8 } & 21.4 & 21.5 & 21.8 \\ 
   & P20 & \textbf{38.2 } & 28.3 & 31.0 & 29.9 \\ 
   & RP & 22.3 & 28.4 & 9.7 & \textbf{30.2 } \\ 
   & SA & \textbf{59.4 } & 29.3 & 20.5 & 30.8 \\ 
   & U & 27.6 & 29.5 & 24.5 & \textbf{30.7 } \\ 
  clayton & B & \textbf{17.7 } & 13.1 & 17.6 & 11.4 \\ 
   & CA & 17.0 & 28.2 & 6.6 & \textbf{31.8 } \\ 
   & P1 & \textbf{20.1 } & 17.6 & 17.6 & 17.3 \\ 
   & P20 & 30.1 & 28.1 & 15.7 & \textbf{31.8 } \\ 
   & RP & 1.7 & 27.5 & 11.6 & \textbf{29.7 } \\ 
   & SA & \textbf{37.5 } & 28.4 & 24.8 & 31.8 \\ 
   & U & 28.2 & 28.1 & 30.8 & \textbf{32.0 } \\ 
  frank & B & 21.8 & 13.6 & \textbf{22.2 } & 12.3 \\ 
   & CA & 5.0 & 32.8 & 6.1 & \textbf{34.7 } \\ 
   & P1 & \textbf{27.6 } & 22.9 & 25.7 & 24.1 \\ 
   & P20 & 32.2 & 31.5 & 18.2 & \textbf{33.0 } \\ 
   & RP & 3.2 & 31.8 & 11.4 & \textbf{33.1 } \\ 
   & SA & 28.7 & 32.7 & 25.4 & \textbf{34.8 } \\ 
   & U & 33.3 & 33.1 & 34.2 & \textbf{34.7 } \\ 
  gumbel & B & \textbf{16.0 } & 12.0 & \textbf{16.0 } & 11.4 \\ 
   & CA & 18.0 & 27.8 & 4.8 & \textbf{30.8 } \\ 
   & P1 & \textbf{29.4 } & 20.9 & 26.3 & 22.3 \\ 
   & P20 & \textbf{31.4 } & 27.3 & 14.6 & 29.8 \\ 
   & RP & 14.8 & 27.9 & 6.8 & \textbf{30.4 } \\ 
   & SA & \textbf{35.1 } & 28.2 & 21.1 & 30.8 \\ 
   & U & 27.9 & 27.9 & 29.5 & \textbf{31.2 } \\ 
  indep & B & 5.7 & 5.9 & 5.6 & \textbf{6.2 } \\ 
   & CA & 4.7 & \textbf{4.9 } & 2.6 & 4.8 \\ 
   & P1 & 4.2 & 4.9 & \textbf{5.1 } & 4.5 \\ 
   & P20 & 4.5 & 3.9 & \textbf{5.7 } & 4.3 \\ 
   & RP & 2.2 & 5.2 & 2.1 & \textbf{5.4 } \\ 
   & SA & 4.7 & \textbf{4.8 } & 3.0 & 4.4 \\ 
   & U & 4.4 & \textbf{4.9 } & 4.6 & 4.8 \\ 
   \hline
\end{tabular}

\begin{tabular}{llll}
  \multicolumn{4}{c}{$N = 100, n = 10$} \\
$\overline{\Mskript}$ & $\overline{\Mskript}_{cop}$ & $dHSIC$ & $dHSIC_{cop}$ \\ 
  \hline
37.6 & 8.6 & \textbf{84.7 } & 62.9 \\ 
  1.7 & 59.5 & 0.3 & \textbf{99.8 } \\ 
  80.8 & 29.7 & 94.5 & \textbf{96.6 } \\ 
  50.5 & 59.4 & 0.0 & \textbf{99.7 } \\ 
  0.2 & 54.2 & 0.0 & \textbf{99.8 } \\ 
  28.5 & 59.1 & 49.9 & \textbf{99.8 } \\ 
  64.4 & 59.2 & 99.7 & \textbf{99.8 } \\ 
  38.7 & 7.6 & \textbf{84.7 } & 63.5 \\ 
  \textbf{100 } & \textbf{100 } & \textbf{100 } & 99.9 \\ 
  \textbf{100 } & \textbf{100 } & 99.6 & 95.9 \\ 
  \textbf{100 } & \textbf{100 } & 93.1 & 99.8 \\ 
  51.4 & \textbf{100 } & 0.0 & 99.8 \\ 
  \textbf{100 } & \textbf{100 } & \textbf{100 } & 99.9 \\ 
  \textbf{100 } & \textbf{100 } & 85.9 & 99.9 \\ 
  40.2 & 6.6 & \textbf{84.4 } & 66.3 \\ 
  97.5 & \textbf{100 } & \textbf{100 } & 99.7 \\ 
  \textbf{100 } & 99.8 & 97.5 & 96.5 \\ 
  \textbf{100 } & \textbf{100 } & 0.0 & 99.8 \\ 
  15.1 & \textbf{100 } & 0.0 & 99.6 \\ 
  \textbf{100 } & \textbf{100 } & \textbf{100 } & 99.7 \\ 
  \textbf{100 } & \textbf{100 } & 96.7 & 99.7 \\ 
  34.9 & 8.5 & \textbf{82.6 } & 60.5 \\ 
  19.7 & 90.1 & 4.2 & \textbf{99.8 } \\ 
  20.9 & 16.3 & 74.8 & \textbf{91.3 } \\ 
  79.1 & 89.0 & 0.0 & \textbf{99.7 } \\ 
  0.0 & 79.0 & 0.0 & \textbf{99.3 } \\ 
  76.4 & 89.9 & 83.7 & \textbf{99.8 } \\ 
  90.3 & 90.0 & 99.2 & \textbf{99.8 } \\ 
  69.9 & 10.0 & \textbf{92.4 } & 74.4 \\ 
  0.1 & 63.4 & 0.4 & \textbf{99.4 } \\ 
  64.5 & 35.9 & 87.8 & \textbf{96.0 } \\ 
  35.4 & 63.7 & 0.0 & \textbf{99.5 } \\ 
  0.0 & 62.2 & 0.0 & \textbf{99.4 } \\ 
  8.9 & 63.5 & 71.4 & \textbf{99.4 } \\ 
  67.7 & 63.4 & 99.0 & \textbf{99.4 } \\ 
  79.6 & 11.4 & \textbf{80.8 } & 61.4 \\ 
  32.2 & 93.8 & 2.3 & \textbf{98.5 } \\ 
  \textbf{95.1 } & 86.7 & 72.4 & 93.6 \\ 
  88.6 & 93.3 & 0.0 & \textbf{98.1 } \\ 
  24.3 & 92.8 & 0.0 & \textbf{98.4 } \\ 
  76.1 & 93.5 & 39.2 & \textbf{98.6 } \\ 
  94.0 & 93.7 & 97.0 & \textbf{98.5 } \\ 
  \textbf{6.9 } & 4.9 & 3.4 & 5.5 \\ 
  0.0 & 2.8 & 0.0 & \textbf{6.7 } \\ 
  5.1 & 4.6 & 1.0 & \textbf{5.8 } \\ 
  4.5 & 3.3 & 0.0 & \textbf{6.4 } \\ 
  0.0 & 3.7 & 0.0 & \textbf{6.5 } \\ 
  2.8 & 2.9 & 0.1 & \textbf{6.7 } \\ 
  4.5 & 2.7 & 3.3 & \textbf{6.8 } \\ 
   \hline
\end{tabular}

\medskip
\caption{Setting of Table \ref{tab:comp} with different $n$.  For each parameter setting the maximum in each row is printed in bold.
\newline
 Observations: With increasing dimension the power increases and the optimal test might change. The tests without the distributional transform can be more powerful than the copula versions. $dHSIC$ becomes liberal with increasing dimension.} 
\label{tab:vary-n}
\end{table}

\begin{table}[ht]
\centering
\begin{tabular}{lllllllllll}
 copula & type & $dH^{0.1}_{cop}$ & $dH^{0.2}_{cop}$ & $dH^{0.5}_{cop}$ & $dH^{0.75}_{cop}$ & $dH^{1}_{cop}$ & $dH^{2}_{cop}$ & $dH^{3}_{cop}$ & $dH^{4}_{cop}$ & $dH^{5}_{cop}$ \\ 
  \hline
normal & B & 9.3 & 21.0 & \textbf{32.6 } & 31.9 & 31.9 & 31.0 & 30.6 & 30.6 & 30.6 \\ 
   & CA & 10.4 & 36.3 & 76.6 & 83.0 & 84.5 & 85.4 & \textbf{85.5 } & \textbf{85.5 } & \textbf{85.5 } \\ 
   & P1 & 8.3 & 26.5 & 59.1 & 63.7 & 64.8 & 64.9 & \textbf{65.1 } & 64.9 & 64.9 \\ 
   & P20 & 9.6 & 35.1 & 75.3 & 81.4 & 83.1 & 84.9 & 85.2 & 85.3 & \textbf{85.4 } \\ 
   & RP & 12.6 & 37.8 & 75.3 & 80.6 & 82.4 & 83.5 & 83.9 & \textbf{84.0 } & \textbf{84.0 } \\ 
   & SA & 11.0 & 35.9 & 76.8 & 82.9 & 84.3 & 85.3 & 85.7 & 85.8 & \textbf{85.9 } \\ 
   & U & 11.3 & 36.1 & 76.7 & 82.4 & 84.6 & 85.4 & 85.4 & \textbf{85.6 } & 85.5 \\ 
  t1 & B & 8.7 & 22.1 & \textbf{31.6 } & 30.8 & 30.8 & 30.2 & 29.8 & 29.6 & 29.5 \\ 
   & CA & 97.3 & \textbf{100 } & 99.9 & 94.7 & 89.1 & 84.9 & 84.5 & 84.5 & 84.4 \\ 
   & P1 & 50.8 & \textbf{83.2 } & 69.6 & 66.3 & 66.7 & 67.7 & 67.7 & 68.1 & 68.1 \\ 
   & P20 & 96.6 & \textbf{100 } & 99.9 & 94.9 & 89.0 & 83.7 & 83.6 & 83.7 & 83.9 \\ 
   & RP & 96.1 & \textbf{100 } & 99.8 & 92.7 & 87.1 & 83.3 & 83.1 & 83.1 & 83.1 \\ 
   & SA & 97.5 & \textbf{100 } & 99.9 & 94.6 & 89.1 & 85.1 & 84.7 & 84.5 & 84.5 \\ 
   & U & 97.3 & \textbf{100 } & 99.9 & 94.8 & 89.1 & 85.1 & 84.6 & 84.3 & 84.3 \\ 
  t3 & B & 9.1 & 20.3 & \textbf{30.2 } & 30.1 & 29.0 & 28.5 & 28.1 & 27.9 & 27.9 \\ 
   & CA & 30.8 & 80.7 & \textbf{86.7 } & 84.1 & 83.5 & 83.9 & 83.9 & 84.1 & 84.4 \\ 
   & P1 & 15.0 & 39.2 & 59.9 & 63.8 & 65.0 & 66.5 & 66.6 & \textbf{66.8 } & 66.7 \\ 
   & P20 & 27.1 & 79.1 & \textbf{86.3 } & 83.2 & 82.5 & 82.9 & 82.9 & 82.9 & 82.9 \\ 
   & RP & 27.8 & 75.3 & \textbf{83.7 } & 81.4 & 81.1 & 80.6 & 80.9 & 81.2 & 81.4 \\ 
   & SA & 31.4 & 80.2 & \textbf{86.4 } & 84.0 & 83.4 & 83.9 & 84.0 & 84.0 & 84.0 \\ 
   & U & 30.4 & 80.4 & \textbf{86.7 } & 84.0 & 83.5 & 83.9 & 83.9 & 84.1 & 84.3 \\ 
  clayton & B & 7.4 & 19.5 & \textbf{31.1 } & 30.9 & 30.4 & 28.2 & 27.9 & 27.9 & 27.9 \\ 
   & CA & 14.8 & 40.6 & 77.5 & 81.7 & 82.7 & 83.7 & 83.8 & \textbf{84.0 } & \textbf{84.0 } \\ 
   & P1 & 9.2 & 24.9 & 48.8 & 51.8 & 52.1 & \textbf{52.5 } & 52.4 & 52.4 & 52.3 \\ 
   & P20 & 13.9 & 39.6 & 76.9 & 81.5 & 82.2 & \textbf{83.0 } & 82.9 & 82.9 & \textbf{83.0 } \\ 
   & RP & 15.5 & 41.2 & 75.6 & 79.7 & 79.8 & 80.8 & 81.1 & \textbf{81.2 } & \textbf{81.2 } \\ 
   & SA & 15.2 & 41.0 & 77.4 & 81.9 & 82.4 & \textbf{83.8 } & \textbf{83.8 } & \textbf{83.8 } & \textbf{83.8 } \\ 
   & U & 14.9 & 40.3 & 77.5 & 81.9 & 82.6 & 83.7 & \textbf{83.9 } & \textbf{83.9 } & \textbf{83.9 } \\ 
  frank & B & 12.0 & 31.7 & \textbf{42.2 } & 41.3 & 40.4 & 39.1 & 38.8 & 38.8 & 38.6 \\ 
   & CA & 16.1 & 56.3 & 84.1 & 85.8 & \textbf{86.3 } & 86.0 & 86.0 & 85.9 & 85.9 \\ 
   & P1 & 11.6 & 36.9 & 64.3 & 66.1 & \textbf{66.3 } & 65.9 & 65.9 & 66.0 & 66.0 \\ 
   & P20 & 16.7 & 56.0 & 83.8 & 85.8 & \textbf{86.0 } & 85.7 & 85.7 & 85.7 & 85.6 \\ 
   & RP & 17.0 & 55.3 & 83.0 & \textbf{85.2 } & 84.6 & 84.2 & 84.3 & 84.4 & 84.4 \\ 
   & SA & 17.2 & 55.5 & 84.3 & 85.8 & \textbf{86.3 } & \textbf{86.3 } & \textbf{86.3 } & \textbf{86.3 } & 86.2 \\ 
   & U & 16.3 & 56.1 & 84.1 & 85.5 & \textbf{86.3 } & 86.1 & 86.0 & 86.2 & 86.2 \\ 
  gumbel & B & 9.3 & 24.3 & \textbf{35.6 } & 33.5 & 32.7 & 31.1 & 30.9 & 30.5 & 30.4 \\ 
   & CA & 26.0 & 48.1 & 75.2 & 80.7 & 81.1 & 82.8 & \textbf{83.0 } & \textbf{83.0 } & \textbf{83.0 } \\ 
   & P1 & 21.3 & 36.1 & 59.8 & 63.6 & 64.3 & 66.0 & 66.2 & 66.2 & \textbf{66.4 } \\ 
   & P20 & 26.1 & 47.9 & 73.5 & 78.7 & 79.7 & 81.4 & \textbf{81.6 } & \textbf{81.6 } & \textbf{81.6 } \\ 
   & RP & 26.5 & 46.2 & 74.4 & 79.5 & 80.7 & 82.4 & 82.6 & \textbf{82.7 } & \textbf{82.7 } \\ 
   & SA & 26.0 & 48.4 & 75.0 & 80.2 & 81.6 & 82.9 & 83.0 & \textbf{83.2 } & \textbf{83.2 } \\ 
   & U & 26.0 & 47.6 & 75.2 & 80.4 & 81.3 & 82.5 & 82.6 & \textbf{82.7 } & \textbf{82.7 } \\ 
  indep & B & \textbf{6.3 } & 5.9 & 6.2 & 5.8 & 5.7 & 6.0 & 6.0 & 6.1 & 6.1 \\ 
   & CA & 3.8 & \textbf{6.3 } & 5.1 & 4.3 & 4.5 & 4.5 & 4.3 & 4.3 & 4.3 \\ 
   & P1 & 5.1 & 5.2 & 5.2 & 4.9 & 5.1 & 5.2 & \textbf{5.3 } & \textbf{5.3 } & \textbf{5.3 } \\ 
   & P20 & 5.1 & \textbf{5.3 } & 5.1 & 4.7 & 4.4 & 4.8 & 4.6 & 4.6 & 4.6 \\ 
   & RP & 4.9 & \textbf{5.7 } & 5.3 & 4.4 & 4.4 & 4.2 & 4.2 & 4.3 & 4.3 \\ 
   & SA & 4.0 & \textbf{6.3 } & 5.1 & 4.3 & 4.7 & 4.7 & 4.7 & 4.4 & 4.4 \\ 
   & U & 3.8 & \textbf{6.0 } & 4.9 & 4.3 & 4.5 & 4.7 & 4.5 & 4.4 & 4.3 \\ 
   \hline
\end{tabular}

\medskip
\caption{Comparison of variants of dHSIC for the setting of Table \ref{tab:comp} with $N=100$ and $n = 5$. The maximum in each row is printed in bold.\newline
 Observations:  The optimal parameter depends on the marginal distribution. In the implementation \cite{PfisPete2019} a heuristic parameter selection is implemented, this yields for uniform marginals a $\delta$ of about $0.2$. Thus the table indicates that in this setting the method is not reliable. It is surprising that for $\delta = 0.2$ and $\delta = 0.5$ the test seems to become liberal.} 
\label{tab:dhsic-variants}
\end{table}

\begin{table}[ht]
\centering
\begin{tabular}{llllll}
  & &\multicolumn{4}{c}{$N = 100, n = 5, \tau = 0.05$} \\
copula & type & $\overline{\Mskript}$ & $\overline{\Mskript}_{cop}$ & $dHSIC$ & $dHSIC_{cop}$ \\ 
  \hline
normal & B & 9.4 & 10.9 & \textbf{13.6 } & 11.4 \\ 
   & CA & 1.3 & 21.4 & 1.3 & \textbf{28.9 } \\ 
   & P1 & \textbf{24.7 } & 15.9 & 15.2 & 18.8 \\ 
   & P20 & 23.5 & 22.2 & 0.0 & \textbf{29.0 } \\ 
   & RP & 0.0 & 20.8 & 0.0 & \textbf{27.7 } \\ 
   & SA & 10.8 & 21.5 & 8.1 & \textbf{28.5 } \\ 
   & U & 22.4 & 21.1 & 25.9 & \textbf{28.7 } \\ 
  t1 & B & 9.7 & 10.1 & \textbf{13.7 } & 12.0 \\ 
   & CA & \textbf{100 } & \textbf{100 } & \textbf{100 } & 46.4 \\ 
   & P1 & \textbf{99.6 } & 71.8 & 46.5 & 29.5 \\ 
   & P20 & \textbf{100 } & \textbf{100 } & \textbf{100 } & 45.9 \\ 
   & RP & 41.6 & \textbf{100 } & 0.1 & 43.8 \\ 
   & SA & \textbf{100 } & \textbf{100 } & 99.0 & 46.6 \\ 
   & U & \textbf{100 } & \textbf{100 } & 14.0 & 46.7 \\ 
  t3 & B & 9.1 & 7.8 & \textbf{11.8 } & 8.5 \\ 
   & CA & 90.7 & 80.1 & \textbf{99.5 } & 36.4 \\ 
   & P1 & \textbf{87.4 } & 36.2 & 21.7 & 24.3 \\ 
   & P20 & \textbf{99.4 } & 77.2 & 49.5 & 36.1 \\ 
   & RP & 11.6 & \textbf{71.3 } & 0.0 & 33.8 \\ 
   & SA & \textbf{99.8 } & 80.1 & 83.6 & 36.4 \\ 
   & U & 79.6 & \textbf{79.9 } & 20.1 & 36.5 \\ 
  clayton & B & 8.3 & 8.7 & \textbf{12.4 } & 12.1 \\ 
   & CA & 5.3 & 23.9 & 3.9 & \textbf{28.5 } \\ 
   & P1 & 11.9 & 11.9 & 10.2 & \textbf{16.3 } \\ 
   & P20 & 25.8 & 22.9 & 0.5 & \textbf{28.6 } \\ 
   & RP & 0.0 & 20.7 & 0.0 & \textbf{25.6 } \\ 
   & SA & 24.3 & 23.3 & 12.3 & \textbf{28.3 } \\ 
   & U & 24.3 & 23.7 & 24.0 & \textbf{28.7 } \\ 
  frank & B & 14.5 & 13.5 & \textbf{16.7 } & 12.5 \\ 
   & CA & 0.9 & 26.7 & 2.0 & \textbf{30.3 } \\ 
   & P1 & \textbf{24.6 } & 19.1 & 14.4 & 19.5 \\ 
   & P20 & 21.7 & 25.2 & 0.2 & \textbf{28.6 } \\ 
   & RP & 0.1 & 24.9 & 0.0 & \textbf{28.2 } \\ 
   & SA & 6.9 & 26.3 & 9.2 & \textbf{29.9 } \\ 
   & U & 26.4 & 26.6 & 26.9 & \textbf{30.5 } \\ 
  indep & B & 4.2 & \textbf{4.8 } & 4.3 & \textbf{4.8 } \\ 
   & CA & 1.0 & 5.0 & 1.5 & \textbf{5.3 } \\ 
   & P1 & 4.3 & 4.7 & 2.8 & \textbf{5.4 } \\ 
   & P20 & \textbf{5.6 } & 5.3 & 0.2 & 5.1 \\ 
   & RP & 0.0 & \textbf{5.6 } & 0.0 & 4.7 \\ 
   & SA & 3.7 & 4.8 & 1.7 & \textbf{4.9 } \\ 
   & U & 5.0 & 4.9 & 4.1 & \textbf{5.1 } \\ 
  gumbel & B & \textbf{14.9 } & 10.8 & 14.1 & 9.8 \\ 
   & CA & 14.8 & \textbf{37.9 } & 2.4 & 30.7 \\ 
   & P1 & \textbf{56.3 } & 31.0 & 13.2 & 23.4 \\ 
   & P20 & \textbf{49.1 } & 37.0 & 0.3 & 30.1 \\ 
   & RP & 10.5 & \textbf{37.8 } & 0.0 & 30.3 \\ 
   & SA & 37.3 & \textbf{38.4 } & 6.6 & 30.4 \\ 
   & U & \textbf{38.2 } & 38.1 & 25.6 & 30.5 \\ 
   \hline
\end{tabular}

\begin{tabular}{llll}
  \multicolumn{4}{c}{$N = 100, n = 5, \tau = 0.2$} \\
$\overline{\Mskript}$ & $\overline{\Mskript}_{cop}$ & $dHSIC$ & $dHSIC_{cop}$ \\ 
  \hline
92.5 & 84.4 & \textbf{97.7 } & 84.0 \\ 
  20.6 & \textbf{100 } & 35.0 & \textbf{100 } \\ 
  \textbf{100 } & 99.1 & \textbf{100 } & 99.5 \\ 
  \textbf{100 } & \textbf{100 } & 26.4 & \textbf{100 } \\ 
  5.9 & \textbf{100 } & 34.9 & \textbf{100 } \\ 
  98.2 & \textbf{100 } & 98.4 & \textbf{100 } \\ 
  \textbf{100 } & \textbf{100 } & \textbf{100 } & \textbf{100 } \\ 
  93.1 & 87.4 & \textbf{98.5 } & 85.4 \\ 
  \textbf{100 } & \textbf{100 } & \textbf{100 } & \textbf{100 } \\ 
  \textbf{100 } & 99.8 & 98.4 & 98.5 \\ 
  \textbf{100 } & \textbf{100 } & \textbf{100 } & 99.9 \\ 
  63.1 & \textbf{100 } & 42.5 & 99.9 \\ 
  \textbf{100 } & \textbf{100 } & \textbf{100 } & \textbf{100 } \\ 
  \textbf{100 } & \textbf{100 } & 99.1 & \textbf{100 } \\ 
  94.6 & 85.9 & \textbf{98.1 } & 84.7 \\ 
  97.1 & \textbf{100 } & \textbf{100 } & \textbf{100 } \\ 
  \textbf{100 } & 99.2 & 99.0 & 99.4 \\ 
  \textbf{100 } & \textbf{100 } & 98.3 & \textbf{100 } \\ 
  28.7 & \textbf{100 } & 34.6 & \textbf{100 } \\ 
  \textbf{100 } & \textbf{100 } & \textbf{100 } & \textbf{100 } \\ 
  \textbf{100 } & \textbf{100 } & \textbf{100 } & \textbf{100 } \\ 
  92.3 & 84.3 & \textbf{96.9 } & 83.4 \\ 
  55.0 & \textbf{100 } & 67.5 & \textbf{100 } \\ 
  98.0 & 97.8 & 97.6 & \textbf{98.5 } \\ 
  99.9 & \textbf{100 } & 44.5 & \textbf{100 } \\ 
  0.3 & \textbf{100 } & 45.6 & \textbf{100 } \\ 
  99.0 & \textbf{100 } & 99.0 & \textbf{100 } \\ 
  \textbf{100 } & \textbf{100 } & \textbf{100 } & \textbf{100 } \\ 
  99.2 & 94.6 & \textbf{99.6 } & 92.4 \\ 
  2.6 & \textbf{100 } & 38.5 & \textbf{100 } \\ 
  99.7 & \textbf{99.8 } & 99.3 & \textbf{99.8 } \\ 
  \textbf{100 } & \textbf{100 } & 55.5 & \textbf{100 } \\ 
  0.7 & \textbf{100 } & 15.9 & \textbf{100 } \\ 
  93.8 & \textbf{100 } & 99.7 & \textbf{100 } \\ 
  \textbf{100 } & \textbf{100 } & \textbf{100 } & \textbf{100 } \\ 
  4.2 & \textbf{4.8 } & 4.3 & \textbf{4.8 } \\ 
  1.0 & 5.0 & 1.5 & \textbf{5.3 } \\ 
  4.3 & 4.7 & 2.8 & \textbf{5.4 } \\ 
  \textbf{5.6 } & 5.3 & 0.2 & 5.1 \\ 
  0.0 & \textbf{5.6 } & 0.0 & 4.7 \\ 
  3.7 & 4.8 & 1.7 & \textbf{4.9 } \\ 
  5.0 & 4.9 & 4.1 & \textbf{5.1 } \\ 
  \textbf{97.8 } & 89.4 & 97.3 & 86.3 \\ 
  52.5 & \textbf{100 } & 56.1 & \textbf{100 } \\ 
  \textbf{99.9 } & 99.6 & 97.9 & 99.7 \\ 
  99.9 & \textbf{100 } & 38.8 & 99.9 \\ 
  43.8 & \textbf{100 } & 5.0 & \textbf{100 } \\ 
  99.3 & \textbf{100 } & 97.5 & \textbf{100 } \\ 
  \textbf{100 } & \textbf{100 } & \textbf{100 } & \textbf{100 } \\ 
   \hline
\end{tabular}

\medskip
\caption{Setting of Table \ref{tab:comp} with $\tau=0.05$ and $0.2.$ For each parameter setting the maximum in each row is printed in bold.
\newline
 Observations: A change of the strength of dependence via Kendall's tau seems not to structurally change the relative performance of the measures. 
} 
\label{tab:vary-tau}
\end{table}

\end{document}